\documentclass[10pt,leqno]{amsart}
   \usepackage[centertags]{amsmath}
   \usepackage{amsfonts}
   \usepackage{amsmath}
   \usepackage{amssymb}
   \usepackage{amsthm}
   \usepackage{newlfont}
   \usepackage[latin1]{inputenc}%%%%para que reconozca las tildes.
\allowdisplaybreaks[3]
\theoremstyle{plain}
\newtheorem{theorem}{Theorem}[section]
\newtheorem{lemma}[theorem]{Lemma}

\newtheorem{corollary}[theorem]{Corollary}

\theoremstyle{definition}

\theoremstyle{remark}

\newtheorem*{remark*}{Remark}

\numberwithin{equation}{section}

\newcommand\CC{{\mathbb C}}

\newcommand\RR{{\mathbb R}}
\newcommand\ZZ{{\mathbb Z}}
\newcommand\NN{{\mathbb N}}
\newcommand\PP{{\mathbb P}}
\newcommand\xx{{\mathbf x}}
\newcommand\zz{{\mathbf z}}
\newcommand\yy{{\mathbf y}}

\newcommand\dgr{\operatorname{deg}}
\newcommand\sign{\operatorname{sign}}

\newcommand*\pFqskip{8mu}
\catcode`,\active
\newcommand*\pFq{\begingroup
        \catcode`\,\active
        \def ,{\mskip\pFqskip\relax}%
        \dopFq
}
\catcode`\,12
\def\dopFq#1#2#3#4#5{%
        {}_{#1}F_{#2}\biggl(\genfrac..{0pt}{}{#3}{#4};#5\biggr)%
        \endgroup
}

\title[Wronskian type determinants of orthogonal polynomials]{Wronskian type determinants of orthogonal polynomials, Selberg type formulas and constant
term identities}

\author{Antonio J. Dur\'an}
\address{A. J. Dur\'an \\
Departamento de An\'{a}lisis Matem\'{a}tico \\
Universidad de Sevilla \\
Apdo (P. O. BOX) 1160\\
41080 Sevilla. Spain.}
\email{duran@us.es }
\thanks{Partially supported by MTM2012-36732-C03-03 (Ministerio de Economía y Competitividad),
FQM-262, FQM-4643, FQM-7276 (Junta de Andalucía) and Feder Funds (European
Union).}



   \date{}
   
\begin{document}

\begin{abstract}
Let $(p_n)_n$ be a sequence of orthogonal polynomials with respect to the measure $\mu$. Let $T$ be a linear operator acting in the linear space
of polynomials $\PP$ and satisfying that $\dgr(T(p))=\dgr(p)-1$, for all polynomial $p$. We then construct a sequence of polynomials $(s_n)_n$,
depending on $T$ but not on $\mu$, such that the Wronskian type $n\times n$ determinant
$\det \left( T^{i-1}(p_{m+j-1}(x))\right)_{i,j=1}^n$ is equal to the $m\times m$ determinant
$\det \left( q^{j-1}_{n+i-1}(x)\right)_{i,j=1}^m$ , up to multiplicative constants, where the polynomials $q_n^i$, $n,i\ge 0$, are defined by
$q_n^i(x)=\sum_{j=0}^n\mu _j^is_{n-j}(x)$, and $\mu_j^i$ are certain generalized moments of the measure $\mu$.
For $T=d/dx$ we recover a Theorem by Leclerc which extends the well-known Karlin and Szeg\H o identities for Hankel determinants whose entries are
ultraspherical, Laguerre and Hermite polynomials. For $T=\Delta$, the first order difference operator, we get some very elegant symmetries for
Casorati determinants of classical discrete orthogonal polynomials. We also show that for certain operators $T$, the second determinant above can
be rewritten in terms of Selberg type integrals, and that for certain operators $T$ and certain families of orthogonal polynomials $(p_n)_n$, one
(or both) of these determinants can also be rewritten as the constant term of certain multivariate Laurent expansions.
\end{abstract}

  \maketitle

\section{Introduction and results}
Determinants whose entries are orthogonal polynomials is a long studied subject. One can mention Turán inequality for Legendre polynomials \cite{Tu} and its generalizations, specially that of Karlin and Szeg\H o on Hankel determinants whose entries are ultraspherical, Laguerre, Hermite, Charlier, Meixner, Krawtchouk and other families of orthogonal polynomials (\cite{KS}). Karlin and Szeg\H o's strategy was to express these Hankel determinants in terms of the Wronskian of certain orthogonal polynomials of another class. However the proofs of Karlin and Szeg\H o did not clearly show what in their formulas resulted from a general
algebraic transformation, and what in contrast was due to some particular
properties of the orthogonal polynomials under consideration. Leclerc (\cite{Le2}) clarified this by
stating that a Wronskian of orthogonal polynomials is proportional to a Hankel determinant whose elements
form a new sequence of polynomials (not necessarily orthogonal). More precisely, let $(a_n)_n$ be a sequence of numbers. With this sequence Leclerc associated two classes of polynomials, $p_n$, $q_n$, $n\ge 0$, defined by
\begin{align}\label{defpnl}
p_n(x)&=\left|\begin{matrix}a_0&a_1&\cdots &a_{n-1}& 1\\
a_1&a_2&\cdots &a_{n}& x\\
\vdots&\vdots& &\vdots&\vdots\\
a_{n-1}&a_n&\cdots &a_{2n-2}&x^{n-1}\\
a_n&a_{n+1}&\cdots &a_{2n-1}& x^n\\
\end{matrix} \right|, \\\label{defqnl}
q_n(x)&=\sum_{j=0}^na_j\binom{n}{j}(-x)^{n-j}.
\end{align}
It is well-known that if $a_n$, $n\ge 0$, are the moments of a non-degenerate positive measure $\mu$, $a_n=\int t^nd\mu$, the polynomials $(p_n)_n$ are then orthogonal with respect to $\mu$.
\begin{theorem}\label{tli} (Theorem 1 of \cite{Le2})
The following identity holds for all integer values of $n$ and $m$, $n,m\ge 1$,
\begin{equation}\label{leci}
\det \left( p^{(i-1)}_{m+j-1}(x)\right)_{i,j=1}^n=C_{n,m}\det \left( q_{n+i+j-2}(x)\right)_{i,j=1}^m ,
\end{equation}
where $C_{n,m}$ is independent of $x$:
$$
C_{n,m}=(-1)^{nm}\prod_{k=1}^{n-1}k!
\det \left( a_{i+j-2}\right)_{i,j=1}^{k+m}.
$$
\end{theorem}
As usual, we take $\prod_{j=n}^mx_j=1$ if $m<n$ and $\det(C)=1$ if the size of $C$ reduces to $0\times 0$.

In \cite{du}, the author has conjectured some identities for Casorati determinants of classical discrete orthogonal polynomials. These conjectures
can be considered as a discrete version of the Karlin and Szeg\H o's identities
for Hankel and Wronskian determinants of ultraspherical, Laguerre and Hermite polynomials.

Here it is one of these conjectures. Let $(\hat c_n^a)_n$, $a\not =0$, be the monic Charlier polynomials (see (\ref{chpol}) below), then for $n,k,m\ge 0$ and $a\in\RR\setminus\{0\}$, we have
\begin{equation}\label{c1.1}
\frac{\det\big(\hat c_{k+j-1}^a(m+i-1)\big)_{i,j=1}^n}{(-a)^{nk}\left( \prod_{j=2}^{n-1}j!\right)}=\frac{\det\big(\hat c_{n+j-1}^{-a}(-k+i-1)\big)_{i,j=1}^m}{a^{nm}\left( \prod_{j=2}^{m-1}j!\right)}.
\end{equation}
In fact, Karlin and Szeg\H o considered in their paper \cite{KS} the determinant in the left hand side of the previous formula (and the corresponding for Meixner, Krawtchouk and Chebyshev), but they were more interested in studying the sign of this determinant that in computing it explicitly.

The purpose of this paper is to show that Leclerc identities for Wronskian of orthogonal polynomials
and author conjectures for Casorati determinants of classical discrete orthogonal polynomials are particular cases of a quite more general result.

Indeed, consider a linear operator $T$ acting in the linear space of polynomials $\PP$ and satisfying that $\dgr(T(p))=\dgr(p)-1$, for all polynomial
$p$. We associate two sequences of polynomials $(r_n)_n$ and $(s_n)_n$ to this operator $T$.
\begin{enumerate}
\item The sequence $(r_n)_n$ satisfies that
\begin{enumerate}
\item $r_0=1$ and the degree of $r_n$ is $n$, $n\ge 0$,
\item $T(r_n)=r_{n-1}$, $n\ge 1$.
\end{enumerate}
It is easy to check that this sequence of polynomials always exists and it is unique if we fix the values of $r_n$, $n\ge 1$, at a given number $x_0$
 (see Section 4 for more details).
\item The sequence $(s_n)_n$ is now defined recursively by $s_0=1$ and
\begin{equation}\label{defsni}
\sum_{j=0}^ns_j(x)r_{n-j}(x)=0.
\end{equation}
\end{enumerate}

We are now ready to establish the main result of this paper.

\begin{theorem}\label{mti}
Consider a linear operator $T$ acting in the linear space of polynomials $\PP$ and satisfying that $\dgr(T(p))=\dgr(p)-1$, for all polynomial
$p$. We associate it the two sequences of polynomials $(r_n)_n$ and $(s_n)_n$ as above.
Let $\mu$ be a measure and consider a sequence $(p_n)_n$ of orthogonal polynomials with respect to $\mu$ (see Section 2 for more details). For a given sequence of polynomials $(\psi_i)_i$, $\psi _i$ of degree $i$, we write $\mu_j^i$, $i,j\ge 0$, for the numbers $\mu_j^i=\int r_j \bar \psi_id\mu$.
We now consider the polynomials $q_n^i$, $i,n\ge 0$,
defined by
\begin{equation}\label{defqny}
q_n^i(x)=\sum_{j=0}^n\mu _j^is_{n-j}(x).
\end{equation}
We then have
\begin{equation}\label{mrby}
\Omega_{m-1}\det \left( T^{i-1}(p_{m+j-1}(x))\right)_{i,j=1}^n=C_{n,m}\det \left( q^{j-1}_{n+i-1}(x)\right)_{i,j=1}^m ,
\end{equation}
where $\Omega_{m-1}$ and $C_{n,m}$ are independent of $x$:
$$
\Omega_{m-1}=\det(\mu_{m-i}^{j-1})_{i,j=1}^{m},\quad C_{n,m}=(-1)^{mn+\binom{m}{2}}\prod_{j=0}^{n-1}\frac{\xi_{m+j}}{\sigma_{m+j}}
$$
and $\xi_n$ and $\sigma_n$ are the coefficients of $x^n$ in the power expansion of $p_n$ and  $r_n$, respectively.
\end{theorem}

Let us note that we have some degrees of freedom in the polynomials $q_n^i$, $n,i\ge 0$. Indeed, for a given number $x_0$, they depend on the numbers $r_n(x_0)$, $n\ge 1$, and on the sequence of polynomials $(\psi_i)_i$.

The content of this paper is as follows. In Section 4 we prove Theorem \ref{mti} (actually a further extension for orthogonal polynomials with respect to a bilinear form defined in $\PP \times \PP $). The proof is similar to that of Leclerc for Theorem \ref{tli} and it is based on some Turnbull determinantal identities (see Section 3). Section 5 is devoted to specialize Theorem \ref{mti} to some relevant operators (such as the derivative, the first order difference operator $\Delta $ and others). We show that for these operators, the polynomials $q_n^i$ and the determinant in the right hand side of (\ref{mrby}) have a richer estructure and more properties. In particular, the specialization for the first order difference operator $\Delta$ will allow us to prove our conjectures in \cite{du} for symmetries of Casorati determinants of Charlier, Meixner, Krawtchouk and Hahn orthogonal polynomials. We also prove some related symmetries for determinants of dual Hahn, Racah and Wilson polynomials. For a convenient normalization of these classical discrete families of orthogonal polynomials the symmetries look very elegant. Here it is the particular case of Hahn polynomials.

For $\alpha+c-N\not =-1, -2, \cdots$ we write $(h_{n}^{\alpha ,c,N})_n$ for the sequence of Hahn polynomials defined by
$$
h_{n}^{\alpha ,c,N}(x)=\frac{1}{n!}\sum_{j=0}^n\frac{(1-N+j)_{n-j}(c+j)_{n-j}}{(n+\alpha +c-N+j)_{n-j}}\binom{n}{j}\binom{x}{j}j!,
$$
and $H_{n,m,x}^{\alpha ,c,N}$, $n,m\ge 0$, for the Casorati Hahn determinant
$$
H_{n,m,x}^{\alpha ,c,N}=\det\big(h_{m+j-1}^{\alpha ,c,N}(x+i-1)\big)_{i,j=1}^n
$$
Then, for $n,m\ge 0$, we have
$$
H_{n,m,x}^{\alpha ,c-n-m,N+n+m}=(-1)^{nm}H_{m,n,-x}^{-\alpha ,2-c,-N}.
$$

In Section 6 we show that for certain operators $T$, the determinant in the right hand side of the previous theorem can be rewritten in terms of Selberg type integrals. Indeed,
write $\Lambda (\xx )$ for the Vandermonde determinant
\begin{equation}\label{defVd}
\Lambda (\xx )=\prod_{1\le i<j\le m}(x_i-x_j),
\end{equation}
where $\xx$ stands for the multivariable $\xx =(x_1,\cdots ,x_m)$. Given a measure $\mu$, consider the multiple integral
\begin{equation}\label{stii}
\int \cdots \int s(\xx )\Lambda^{2\gamma}(\xx )d\mu(\xx ),
\end{equation}
where $d\mu(\xx )=d\mu(x_1)\cdots d\mu(x_m)$ and $s(\xx )$ is a symmetric polynomial in the variables $x_1,\cdots , x_m$. It seems to be rather difficult to get explicit fomulae for the integral (\ref{stii}) even for simple measures $\mu$.

When $\mu$ is the Jacobi  measure $d\mu =x^{\alpha -1}(1-x)^{\beta -1} dx$, $ \alpha , \beta >0$, and $s(\xx )=1$, the integral (\ref{stii}) is the celebrated Selberg integral (\cite{Se}, see also \cite{FW}):
$$
\int _0^1 \cdots \int _0^1 \Lambda^{2\gamma}(\xx ) \prod_{j=1}^mx_j^{\alpha -1}(1-x_j)^{\beta -1}d\xx=\prod_{j=0}^{m-1}\frac{\Gamma(\alpha+j\gamma)\Gamma(\beta+j\gamma)
\Gamma(1+(j+1)\gamma)}{\Gamma(\alpha+\beta +(m+j-1)\gamma)\Gamma(1+\gamma)}.
$$
Other choices of $s(\xx )$ give well-known extensions of the Selberg integral. For instance, $s(\xx )=\prod_{j=1}^m(x_j-u)$ is Aomoto's integral \cite{Ao}, and $s(\xx )=J_\lambda(x_1,\cdots, x_m,1/\gamma)$, where $\gamma \in \NN $ and
$J_\lambda$ is a Jack polynomial, is Kadell's integral \cite{Ka1}.

For $s(\xx )=1$ and $\mu$ the Hermite measure $d\mu =e^{-x^2/2}dx$, the integral (\ref{stii}) is the also celebrated Mehta integral.

Given a number $u$ and a linear operator $T$ as in Theorem \ref{mti}, we associate it the two sequences of polynomials $(r_n)_n$ and $(s_n)_n$ as above.
In this case we take $r_n(u)=0$, $n\ge 1$ (these conditions define $r_n$, $n\ge 0$, uniquely). We write $r_n=r_{n,u}$ to stress the dependence of $r_n$ on $u$. We then have.

\begin{theorem}\label{tise} Assume that there are numbers $a_j$, $j\ge 1$, with $a_1=u$ and  $r_{n,u}(x)=\sigma_n\prod_{j=1}^n(x-a_j)$, $n\ge 1$. Let $(\hat p_n)_n$ be the sequence of monic orthogonal polynomials with respect to the measure $\mu$, then
\begin{equation}\label{hfpi}
\int \cdots \int \Lambda ^2(\xx)\prod_{j=1}^{m}r_{n,u}(x_j)
d\mu(\xx )=C_{T,n,m}\det\left( T^{i-1}(\hat p_{m+j-1}(u))\right)_{i,j=1}^n,
\end{equation}
where
$$
C_{T,n,m}=(-1)^{mn}m!\sigma_n^m\prod_{j=0}^{m-1}\Vert \hat p_j\Vert^2\prod_{j=0}^{n-1}\sigma_{j}.
$$
\end{theorem}

For $n=1$, the identity (\ref{hfpi}) is known (see \cite{Sz}, (2.2.10)).

Theorem \ref{tise} applies to the  relevant operators $T=d/dx$ and $T=\Delta$   since in each case we can choose polynomials $r_n$, $n\ge 0$, with $T(r_n)=r_{n-1}$ and satisfying the hypothesis in Theorem \ref{tise}.
We can then deduce some Selberg type integrals and sums corresponding to $\gamma =1$ in (\ref{stii}). For instance, for $T=d/dx$, $n=0$, $u=0$, and $\mu$ the Jacobi weight, we recover the Selberg integral for $\gamma=1$. The case $n=1$ and again $\gamma=1$ is Aomoto's result for (\ref{stii}) when $s(\xx )=\prod_{j=1}^m(x_j-u)$ (\cite{Ao}). For $T=\Delta$, $n=0$, $u=0$ and $\mu$ the Meixner, Krawtchouk or Hahn weights
Theorem \ref{tise} gives a deduction for Askey conjectures 5, 6 and 7, respectively, in \cite{Ask} when $\gamma=1$. The case $n=1$ and again $\gamma=1$ can be considered a discrete version of Aomoto's integral.

\bigskip
In Section 7, we show that for certain operators $T$ and certain families of orthogonal polynomials $(p_n)_n$, one (or both) of the determinants in (\ref{mrby}) can be rewritten as the constant term of certain multivariate Laurent expansions. This is not surprising, because certain  identities
for constant term of Laurent polynomials are closely related to Selberg integrals. The first of these constant term identities was conjectured by Dyson \cite{dy} in 1962 in his study of the statistical theory of complex systems:
$$
\mbox{C.T.$\vert _{\zz =0}$} \prod_{1\le i<j\le n}\left(1-\frac{z_i}{z_j}\right)^k\left(1-\frac{z_j}{z_i}\right)^k=\frac{(kn)!}{(k!)^n},
$$
where we denote by $\mbox{C.T.$\vert _{\zz =0}$}F(\zz )$ the constant term of the Laurent expansion at $\zz =(z_1,\cdots,z_n)=(0,\cdots ,0)$ of $F(\zz )$.

Macdonald \cite{Mac} in 1982 related Dyson identity with the root system $A_{n-1}$ and posed some other conjectures which generalize Dyson constant term identity to all finite root systems. Morris \cite{Mo} also proposed some constant term identities related to root systems and Selberg integrals. The following one can be considered a contour version of Selberg integral:
\begin{equation}\label{mocti}
\mbox{C.T.$\vert _{\zz =0}$}\Lambda ^{2k}(\zz )
\prod_{j=1}^n\frac{\left(1-\frac{z_j}{a}\right)^x}{z_j^{m+k(n-1)}}=(-1)^{k\binom{n}{2}+mn}a^{-mn}
\prod_{j=0}^{n-1}\frac{(x+kj)!(k(j+1))!}{(x-m+kj)!(m+kj)!k!},
\end{equation}
where $x,m,k,n$ are nonnegative integers, $m\le x$ and $a\not =0$.

For $T=\Delta$ and $\mu$ the Meixner measure
$$
\mu=\sum_{x=0}^\infty \frac{a^x\Gamma (x+c)}{x!}\delta _x,
$$
Theorem \ref{mti} gives a generalization of Morris identity (\ref{mocti}) for $k=1$. Indeed, using the integral representation for the Meixner polynomials $(m_n^{a,c})_n$ provided by the generating function
$$
\left( 1-\frac{z}{a}\right) ^x(1-z)^{-x-c}=\sum_{n=0}^\infty m_n^{a,c}(x)z^n,
$$
we prove the following constant term identity:
\begin{align*}
\mbox{C.T.$\vert _{\zz =0}$} \prod_{j=1}^n&\frac{\left( 1-\frac{z_j}{a}\right)^{x}}{(1-z_j)^{x+c+n-1}z_j^{m+n-1}}\prod_{1\le i<j\le n}(z_i-z_j)^2\\ &=(-1)^{\binom{m}{2}+\binom{n}{2}+nm}n!a^{\binom{m}{2}} \det\left(m_{n+i-1}^{a,-c-n-i-j+3}(-x+j-1) \right)_{i,j=1}^m.
\end{align*}
For $c=-x-n+1$, the determinant in the right hand side of this identity can be easily computed, and we recover Morris identity (\ref{mocti}) for $k=1$.
For $m=1$, we get
\begin{align*}
\mbox{C.T.$\vert _{\zz =0}$} \prod_{j=1}^n\frac{\left( 1-\frac{z_j}{a}\right)^{x}}{(1-z_j)^{x+c+n-1}z_j^{n}}\prod_{1\le i<j\le n}(z_i-z_j)^2=(-1)^{\binom{n+1}{2}}n!m_{n}^{a,-c-n+1}(-x).
\end{align*}
which it can be considered a contour integral version of Aomoto's integral \cite{Ao}.

We find other constant term identities using ultraspherical and Charlier polynomials.

\section{Preliminaries}
Let $\langle ,\rangle$ be a bilinear form acting in the linear space of polynomials: $\langle ,   \rangle :\PP \times \PP \to \CC $. Let $K$ be a positive integer or infinity. We say that the polynomials $p_n$, $0\le n\le K$, are pseudo-orthogonal with respect to $\langle , \rangle$ if
$\langle p_n,p_j\rangle =0$, $0\le j\le n-1$ and $1\le n\le K$. If, in addition, $p_n$ has degree $n$, $0\le n\le K$, we say that
$(p_n)_{n=0}^K$ is a sequence of orthogonal polynomials with respect to $\langle ,\rangle$.
It is clear that the $n$-th orthogonal polynomial
with respect to a bilinear form, if there exists, is unique up to multiplicative constants.

Let $(r_n)_{n=0}^K$ and $(\psi _n)_{n=0}^K$ be two sequences of polynomials satisfying that $\dgr (r_n)=\dgr (\psi_n)=n$.
Write $\sigma _n$ for the leading coefficient of $r_n$. Write also
$\mu_j^i=\langle r_j,\psi_i\rangle$, $0\le i,j\le K$. We can then construct pseudo orthogonal polynomials $p_n$, $n=0,\cdots ,K$, with respect to $\langle ,\rangle$ using the formula
\begin{equation}\label{fdpo1}
p_n(x)=\left| \begin{matrix} \mu _n^0 & \mu _n^1&\cdots & \mu _n^{n-1}&r_n(x)\\
\mu _{n-1}^0 &\mu_{n-1}^1& \cdots &\mu_{n-1}^{n-1}&r_{n-1}(x)\\ \vdots &\vdots & \ddots & \vdots& \vdots
\\\mu_0^0 & \mu_0^1&\cdots &\mu_0^{n-1}&1 \end{matrix} \right|,\quad 1\le n\le K, \quad p_0=1.
\end{equation}
Indeed, from the definition, we straightforwardly have that $\langle p_n,\psi_j\rangle=0$, $j=0,\cdots , n-1$. Since $\dgr(r_j)=j$, we have that $\dgr(p_n)\le n$, and since $\dgr(\psi_j)=j$ as well, we can write each $p_j$ as a linear combination of $\psi_i$, $i=0,\cdots, j$. As a consequence we get
$\langle p_n,p_j\rangle=0$, $j=0,\cdots, n-1$.

Consider now the determinants $\Omega_n$, $n\ge 0$, defined by
\begin{equation}\label{defome}
\Omega _n=\left| \begin{matrix} \mu _n^0 & \cdots & \mu _n^n\\ \vdots & \ddots & \vdots \\ \mu _0^0 & \cdots &\mu_{0}^n \end{matrix} \right|.
\end{equation}
It is clear from (\ref{fdpo1}) that the coefficient of $x^n$ in the power expansion of $p_n$ is $(-1)^{n}\Omega_{n-1}\sigma_n$.

Hence, the bilinear form $\langle ,\rangle$ has a sequence $p_n$, $n=0,\cdots ,K$, of orthogonal polynomials if and only if $\Omega _n\not =0$, $0\le n\le K$. We can then construct a sequence of monic orthogonal polynomials $(\hat p_n)_{n=0}^K$, $\hat p_n$ of
degree $n$, using the formula
\begin{equation}\label{fdpop}
\hat p_n(x)=\frac{(-1)^{n}}{\Omega _{n-1}\sigma_n}\left| \begin{matrix} \mu _n^0 & \mu _n^1&\cdots & \mu _n^{n-1}&r_n(x)\\
\mu _{n-1}^0 &\mu_{n-1}^1& \cdots &\mu_{n-1}^{n-1}&r_{n-1}(x)\\ \vdots &\vdots & \ddots & \vdots& \vdots
\\\mu_0^0 & \mu_0^1&\cdots &\mu_0^{n-1}&1 \end{matrix} \right|,\quad 1\le n\le K, \quad p_0=1.
\end{equation}

The determinant $\Omega_n$ can be computed in terms of the norm of the monic polynomials. Indeed, write $\Vert \hat p_n\Vert ^2=\langle \hat p_n,\hat p_n\rangle $.
If we write $\upsilon_n$ for the leading coefficient of $\psi_n$, we have $\psi_n=\upsilon _n \hat p_n+\sum_{j=0}^{n-1} \alpha_{n,j}\hat p_j$. And hence $\langle \hat p_n,\psi_n\rangle=\bar \upsilon _n \Vert \hat p_n\Vert ^2$. Using (\ref{fdpop}), we also get $\langle \hat p_n,\psi_n\rangle=(-1)^n\Omega_n/(\Omega_{n-1}\sigma_n)$. From where we get
\begin{equation}\label{oitnm}
\Omega_n=(-1)^{n(n+1)/2}\prod_{j=0}^n\sigma_j\bar\upsilon_j\Vert \hat p_j\Vert ^2.
\end{equation}
In this paper, we are mainly interested in bilinear form defined by a measure $\mu$ supported in the real line: $\langle p,q\rangle =\int p\overline{q}d\mu$.

A bilinear form can be represented by a measure supported in the real line if and only if it is Hermitian ($\langle p,q\rangle =\overline
{\langle q,p\rangle}$, $p,q\in \PP$) and the operator of multiplication by $x$ is symmetric with respect
to the bilinear form: $\langle xp,q\rangle =\langle p,xq\rangle$, $p,q\in \PP$.

\bigskip
We need some combinatorial identities which we include in the following lemma.

\begin{lemma}
For $x, u\in\RR $ and $n,i,j,g\in \NN$  we have
\begin{equation}\label{abc2}
\binom{x}{i}=\sum_{l=0}^i\binom{j}{i-l}\binom{x-j}{l}.
\end{equation}
For $i\le g\le n+i$, we have
\begin{equation}\label{abc1}
\binom{g}{i}\binom{-x+i}{n+i-g}=\sum_{j=0}^{\min(n,g)}\binom{j}{i+j-g}\binom{g}{g-j}\binom{-x}{n-j}.
\end{equation}
And for $0\le g\le n$, we have
\begin{align}\label{abc3}
\binom{g+i}{i}&\binom{x+u+n+i-1}{n-g}(x-i)_{n-g}\\\nonumber &
=\sum_{j=0}^{n}(-1)^{j-g}(u+g+i)_{j-g}(x)_{n-j} \binom{j}{j-g}\binom{g+i}{j}\binom{x+u+n-1}{r-j}.
\end{align}
\end{lemma}

\begin{proof}

The identities can be   proved from the basic identity $\binom{u}{v}=\binom{u-1}{v}+\binom{u-1}{v-1}$
using induction on $x$.

\end{proof}

\section{Leclerc version of Turnbull determinantal identities}
In this section, we recall briefly Leclerc notation of \cite{Le1} for minor identities (which is a variant of Turnbull's dot notation, \cite{Tur}, p. 27). Let $M$ be a $n\times p$ matrix with $p>n$ and $a,b,c,\cdots, d$ be $n$ column vectors of $M$. The maximal minor of $M$ taken on these $n$ columns is denoted by either a bracket or a one line tableau:
$$
[ab\cdots c]=\framebox{$\begin{matrix} a& b& \cdots &c\end{matrix}$}.
$$
A product of $k$ minors of $M$ is designated by a $k\times n$ tableau. For instance, for $k=2$:
$$
[ab\cdots c]\cdot [de\cdots f]=\framebox{$\begin{matrix}a& b& \cdots &c\\ d& e& \cdots &f\end{matrix}$}.
$$
To denote alternating sums of products of minors, tableaux with boxes enclosing certain vectors are used. Let $T$ be a $k\times n$ tableau and $A$ a subset of elements of $T$. Given a permutation $\sigma$ of the elements of $A$, denote by $\sigma (T)$ the tableau in which the elements of $A$ are permuted by $\sigma$. Now, boxing in $T$ the elements of $A$, we get a new tableau $\tau$ which will represent the alternating sum of all tableaux $\sigma(T)$, taking into account the fact that a permutation of elements of the same row gives a trivial action. More precisely $\tau$ is defined by
$$
\tau=\sum_{\sigma}\sign(\sigma)\sigma(T),
$$
where $\sigma $ runs trough the cosets of the symmetric group $\Sigma(A)$ modulo the subgroup of the permutations which leave unchanged the rows of $T$. For example,
$$
\framebox{$\begin{matrix} \framebox{$a$}& \framebox{$b$} & c\\d&e&\framebox{$f$}\end{matrix}$}=\framebox{$\begin{matrix} a& b & c\\d&e&f\end{matrix}$}-\framebox{$\begin{matrix} f& b & c\\d&e&a\end{matrix}$}-\framebox{$\begin{matrix} a& f & c\\d&e&b\end{matrix}$}.
$$

We can now state Turnbull's identity (see \cite{Tur}, p. 48 and \cite{Le1}).

\begin{theorem}\label{lvti}
Let $\tau$ be a $p\times n$ tableau with set of enclosed elements $A$ of cardinality $\le n$. Let $R$ be one of the rows of $T$ and denote by $B$ the set of elements of $R$ which are not enclosed. Form a new tableau $\nu$ by (i) exchanging the elements of $A$ which do not belong to $R$ with elements of $B$; (ii) removing the boxes of the elements of $A$; (iii) boxing the elements of $B$; then $\tau=\nu$.
\end{theorem}

Take for instance
$$
\tau=\framebox{$\begin{matrix} \framebox{$\mathbf{a}$}& \alpha &\beta &\gamma &\delta &\epsilon \\
\framebox{$\mathbf{b}$} & \framebox{$\mathbf{c}$}&f&g&h&i\\\framebox{$\mathbf{d}$}&j&k&l&m&o\end{matrix}$}
$$
and choosing for $R$ the first row, we have $A=\{\mathbf{a}, \mathbf{b}, \mathbf{c}, \mathbf{d}\}$, $B=\{\alpha, \beta, \gamma, \delta ,\epsilon\}$ and we obtain
$$
\tau=\framebox{$\begin{matrix} \framebox{$\mathbf{a}$}& \alpha &\beta &\gamma &\delta &\epsilon \\
\framebox{$\mathbf{b}$} & \framebox{$\mathbf{c}$}&f&g&h&i\\\framebox{$\mathbf{d}$}&j&k&l&m&o\end{matrix}$}=
\framebox{$\begin{matrix} \mathbf{a}& \mathbf{b} &\mathbf{c} &\mathbf{d} &\framebox{$\delta$} &\framebox{$\epsilon$} \\
\framebox{$\alpha$} & \framebox{$\beta$}&f&g&h&i\\\framebox{$\gamma$}&j&k&l&m&o\end{matrix}$}=\nu
$$
Note that here $\tau$ represents a sum of $12$ products of minors, whereas $\nu$ stands
for a sum of 30 products.

\section{Wronskian type determinants of orthogonal polynomials}

We are now ready to establish and prove the main result of this paper.

Consider a linear operator $T$ acting in the linear space of polynomials $\PP$ and satisfying that $\dgr(T(p))=\dgr(p)-1$, for all polynomial
$p$. We associate two sequences of polynomials $(r_n)_n$ and $(s_n)_n$ to this operator $T$. The sequence $(r_n)_n$ satisfies that
\begin{equation}\label{defrn}
\begin{cases} &\mbox{$r_0=1$ and the degree of $r_n$ is $n$, $n\ge 0$},\\
 &\mbox{$T(r_n)=r_{n-1}$, $n\ge 0$}. \end{cases}
\end{equation}
This sequence of polynomials always exists and it is unique if we fix the values of $r_n$, $n\ge 1$, at a given number $x_0$. Indeed, fix a sequence of numbers $(\xi_n)_{n\ge 1}$ and assume that we have constructed the polynomials $r_j$, $j=0,\cdots, n$, satisfying $r_j(x_0)=\xi_j$ and (\ref{defrn}). Since $T((x-x_0)^{n+1})$ is a polynomial of degree $n$, we can write $T((x-x_0)^{n+1})=\sum_{j=0}^n\alpha_jr_j$, with $\alpha_n\not =0$. We
then define
$$
r_{n+1}(x)=\frac{(x-x_0)^{n+1}-\sum_{j=0}^{n-1}\alpha_j(r_{j+1}(x)-\xi_{j+1})}{\alpha_n}+\xi_{n+1}.
$$
From where we easily get that $r_{n+1}$ has degree $n+1$, $T(r_{n+1})=r_n$ and $r_{n+1}(x_0)=\xi_{n+1}$.

The sequence $(s_n)_n$ are now defined recursively by $s_0=1$ and
\begin{equation}\label{defsn}
\sum_{j=0}^ns_j(x)r_{n-j}(x)=0.
\end{equation}
It is easy to see that if we write $\Psi_r(x,t)$, $\Psi_s(x,t)$ for the (formal) generating functions of the sequences $(r_n)_n$, $(s_n)_n$, respectively,
\begin{equation}\label{gfdi}
\Psi_r(x,t)=\sum_{n=0}^\infty r_n(x)t^n,\quad \Psi_s(x,t)=\sum_{n=0}^\infty s_n(x)t^n,
\end{equation}
we have $\Psi_r(x,t)\Psi_s(x,t)=1$.

\begin{theorem}\label{mth} Let $T$ be a linear operator acting in the linear space of polynomials satisfying $\dgr(T(p))=\dgr(p)-1$, for all polynomial
$p$. We associate it the two sequences of polynomials $(r_n)_n$ and $(s_n)_n$ defined as above (see (\ref{defrn}) and (\ref{defsn})). Let $\langle ,\rangle$ be a bilinear form acting in the linear space of polynomials.
For a given sequence of polynomials $(\psi_i)_i$, $\psi _i$ of degree $i$, we write $\mu_j^i=\langle r_j,\psi_i\rangle$, $0\le i,j$. With the operator $T$ and the bilinear form $\langle ,\rangle$ we associate the
polynomials $p_n, q_n^i$, $n,i\ge0$. The polynomials $p_n$, $n\ge0$, are the pseudo orthogonal polynomials with respect to
$\langle ,\rangle$ defined by $p_0=1$ and for $n\ge 1$
\begin{equation}\label{fdpo2}
p_n(x)=\left| \begin{matrix} \mu _n^0 & \mu _n^1&\cdots & \mu _n^{n-1}&r_n\\
\mu _{n-1}^0 &\mu_{n-1}^1& \cdots &\mu_{n-1}^{n-1}&r_{n-1}\\ \vdots &\vdots & \ddots & \vdots& \vdots
\\\mu_0^0 & \mu_0^1&\cdots &\mu_0^{n-1}&1 \end{matrix}\right| .
\end{equation}
The polynomials $q_n^i$, $n,i\ge0$, are defined by
\begin{equation}\label{defqn}
q_n^i(x)=\sum_{j=0}^n\mu _j^is_{n-j}(x).
\end{equation}
We then have
\begin{equation}\label{mrb}
\det\left(T^{i-1}(p_{m+j-1})\right)_{i,j=1}^n=C_{n,m}\det\left( q^{j-1}_{n+m-i}\right)_{i,j=1}^m,
\end{equation}
where $C_{n,m}=(-1)^{\binom{n}{2}}\prod_{j=0}^{n-2}\Omega_{m+j}$ is independent of $x$ (and $\Omega_n$ is defined by (\ref{defome})).
\end{theorem}

\begin{proof}

The proof is similar to that of Leclerc for Theorem \ref{tli} in the Introduction (see \cite{Le2}).

Let us introduce the $(m+n)\times(m+3n-2)$ matrix
$$
\left(\begin{smallmatrix}
1&0&\cdots &0&\mu^0_{m+n-1}&\mu^1_{m+n-1}&\cdots &\mu^{m+n-2}_{m+n-1}&r_{m+n-1}&T(r_{m+n-1})&\cdots &T^{n-1}(r_{m+n-1})\\
0&1&\cdots &0&\mu^0_{m+n-2}&\mu^1_{m+n-2}&\cdots & \mu^{m+n-2}_{m+n-2}&r_{m+n-2}&T(r_{m+n-2})&\cdots &T^{n-1}(r_{m+n-2}) \\
\vdots&\vdots& &\vdots&\vdots &\vdots && \vdots&\vdots& \vdots & &\vdots\\
0&0&\cdots &1&\mu^0_{m+1}&\mu^1_{m+1}&\cdots & \mu^{m+n-2}_{m+1}&r_{m+1}&T(r_{m+1})&\cdots &T^{n-1}(r_{m+1})
\\
0&0&\cdots &0&\mu^0_{m}&\mu^1_{m}&\cdots & \mu^{m+n-2}_{m}&r_{m}&T(r_{m})&\cdots &T^{n-1}(r_{m})\\
\vdots&\vdots& &\vdots&\vdots &\vdots && \vdots&\vdots& \vdots & &\vdots\\
0&0&\cdots &0&\mu^0_{0}&\mu^1_{0}&\cdots & \mu^{m+n-2}_{0}&1&0&\cdots &0
\end{smallmatrix} \right).
$$
The column vectors of this matrix will be denoted from left to right by:
$$
\mathbf{1},\mathbf{2},\cdots , \mathbf{n-1},\mu^0,\mu^1\cdots, \mu^{m+n-2},\mathbf{r},\mathbf{T(r)},\cdots, \mathbf{T^{n-1}(r)}.
$$
Using the notation of Section 3, the left-hand side of (\ref{mrb}) is written as
$$
\left|\begin{matrix}[\mathbf{1} \mathbf{2}\cdots \mathbf{n-1}\mu^0\cdots \mu^{m-1}\mathbf{r}] &
[\mathbf{1} \mathbf{2}\cdots \mathbf{n-2}\mu^0\cdots \mu^{m}\mathbf{r}]&\cdots &[\mu^0\cdots \mu^{m+n-2}\mathbf{r}] \\
[\mathbf{1} \mathbf{2}\cdots \mathbf{n-1}\mu^0\cdots \mu^{m-1}\mathbf{T(r)}] &
[\mathbf{1} \mathbf{2}\cdots \mathbf{n-2}\mu^0\cdots \mu^{m}\mathbf{T(r)}]&\cdots &[\mu^0\cdots \mu^{m+n-2}\mathbf{T(r)}]\\
\vdots&\vdots& &\vdots \\
[\mathbf{1} \cdots \mathbf{n-1}\mu^0\cdots \mu^{m-1}\mathbf{T^{n-1}(r)}] &
[\mathbf{1} \cdots \mathbf{n-2}\mu^0\cdots \mu^{m}\mathbf{T^{n-1}(r)}]&\cdots &[\mu^0\cdots \mu^{m+n-2}\mathbf{T^{n-1}(r)}]
\end{matrix}\right|
$$
Now using the transformation of Theorem \ref{lvti} with $R$ being the last row, this tableau is equal to
\begin{align*}
&\framebox{$\begin{matrix} \mathbf{1}& \mathbf{2} &\cdots &
\mathbf{n-2} &\mathbf{n-1} &\mu^0 & \mu^1 & \cdots& \mu^{m-1}&\framebox{$\mu^m$} \\
\mathbf{1}& \mathbf{2} &\cdots &
\mathbf{n-2} &\mu^0 & \mu^1 &\mu^2& \cdots& \mu^{m}&\framebox{$\mu^{m+1}$}\\
\vdots &\vdots && & & & && \vdots &\vdots \\
\mathbf{1}& \mu^0 &\cdots &
\cdots &\cdots & \cdots &\cdots& \cdots& \mu^{m+n-3}&\framebox{$\mu^{m+n-2}$}\\
\framebox{$\mu^0$}& \framebox{$\mu^{1}$} &\cdots &
\cdots &\framebox{$\mu^{m-1}$} & \mathbf{r} &\mathbf{T(r)}& \cdots& \mathbf{T^{n-2}(r)}&\mathbf{T^{n-1}(r)}
\end{matrix}$}\\
&\hspace{1cm}=
\framebox{$\begin{matrix} \mathbf{1}& \mathbf{2} &\cdots &
\mathbf{n-2} &\mathbf{n-1} &\mu^0 & \mu^1 & \cdots& \mu^{m-1}&\mu^m \\
\mathbf{1}& \mathbf{2} &\cdots &
\mathbf{n-2} &\mu^0 & \mu^1 &\mu^2& \cdots& \mu^{m}&\mu^{m+1}\\
\vdots &\vdots && & & & && \vdots &\vdots \\
\mathbf{1}& \mu^0 &\cdots &
\cdots &\cdots & \cdots &\cdots& \cdots& \mu^{m+n-3}&\mu^{m+n-2}\\
\mu^0& \mu^{1} &\cdots &
\cdots &\mu^{m-1} & \mathbf{r} &\mathbf{T(r)}& \cdots& \mathbf{T^{n-2}(r)}&\mathbf{T^{n-1}(r)}
\end{matrix}$}
\end{align*}
Note that all boxes have been deleted in the second tableau because all the
other permutations of the letters enclosed give rise to tableaux with two equal
letters on some row, which are therefore equal to zero in view of the skew-symmetry of the determinant. Thus, we have rewritten the Wronskian type determinant of (\ref{mrb}) as a product of $n$ determinants, where only the last one depends on the variable $x$. Explicitly, we have obtained that the left-hand side of (\ref{mrb}) is equal to
$$
\Omega_{m+n-2}\Omega_{m+n-3}\cdots \Omega _m\Theta(x)
$$
where
$$
\Theta(x)=\left|\begin{matrix}
\mu^0_{m+n-1}&\mu^1_{m+n-1}&\cdots &\mu^{m-1}_{m+n-1}&r_{m+n-1}&T(r_{m+n-1})&\cdots &T^{n-1}(r_{m+n-1})\\
\mu^0_{m+n-2}&\mu^1_{m+n-2}&\cdots &\mu^{m-1}_{m+n-2}&r_{m+n-2}&T(r_{m+n-2})&\cdots &T^{n-1}(r_{m+n-2})\\
\vdots &\vdots&& \vdots &\vdots&\vdots& & \vdots \\
\mu^0_{1}&\mu^1_{1}&\cdots &\mu^{m-1}_{1}&r_{1}&1&\cdots &0\\
\mu^0_{0}&\mu^1_{0}&\cdots &\mu^{m-1}_{0}&1&0&\cdots &0
\end{matrix}\right|,
$$
and $\Omega _n$ is defined by (\ref{defome}).
Taking into account that $T(r_n)=r_{n-1}$, we have
$$
\Theta(x)=\left|\begin{matrix}
\mu^0_{m+n-1}&\mu^1_{m+n-1}&\cdots &\mu^{m-1}_{m+n-1}&r_{m+n-1}&r_{m+n-2}&\cdots &r_{m}\\
\mu^0_{m+n-2}&\mu^1_{m+n-2}&\cdots &\mu^{m-1}_{m+n-2}&r_{m+n-2}&r_{m+n-3}&\cdots &r_{m-1}\\
\vdots &\vdots&& \vdots &\vdots&\vdots& & \vdots \\
\mu^0_{1}&\mu^1_{1}&\cdots &\mu^{m-1}_{1}&r_{1}&1&\cdots &0\\
\mu^0_{0}&\mu^1_{0}&\cdots &\mu^{m-1}_{0}&1&0&\cdots &0
\end{matrix}\right|.
$$
We now add to the first row of this determinant the linear combination of the next $m+n-1$ rows
multiplying the $j$-th row by $s_{j-1}(x)$, $j=2,\cdots, m+n$. (\ref{defqn}) and (\ref{defsn}) then give
$$
\Theta(x)=\left|\begin{matrix}
q^0_{m+n-1}&q^1_{m+n-1}&\cdots &q^{m-1}_{m+n-1}&0&0&\cdots &0\\
\mu^0_{m+n-2}&\mu^1_{m+n-2}&\cdots &\mu^{m-1}_{m+n-2}&r_{m+n-2}&r_{m+n-3}&\cdots &r_{m-1}\\
\vdots &\vdots&& \vdots &\vdots&\vdots& & \vdots \\
\mu^0_{1}&\mu^1_{1}&\cdots &\mu^{m-1}_{1}&r_{1}&1&\cdots &0\\
\mu^0_{0}&\mu^1_{0}&\cdots &\mu^{m-1}_{0}&1&0&\cdots &0
\end{matrix}\right|.
$$
In a similar way, we add to the second row of this determinant the linear combination of the next $m+n-2$ rows
multiplying the $j$-th row by $s_{j-2}(x)$, $j=3,\cdots, m+n$. (\ref{defqn}) and (\ref{defsn}) then give
$$
\Theta(x)=\left|\begin{matrix}
q^0_{m+n-1}&q^1_{m+n-1}&\cdots &q^{m-1}_{m+n-1}&0&0&\cdots &0\\
q^0_{m+n-2}&q^1_{m+n-2}&\cdots &q^{m-1}_{m+n-2}&0&0&\cdots &0\\
\mu^0_{m+n-3}&\mu^1_{m+n-3}&\cdots &\mu^{m-1}_{m+n-3}&r_{m+n-3}&r_{m+n-4}&\cdots &r_{m-2}\\
\vdots &\vdots&& \vdots &\vdots&\vdots& & \vdots \\
\mu^0_{1}&\mu^1_{1}&\cdots &\mu^{m-1}_{1}&r_{1}&1&\cdots &0\\
\mu^0_{0}&\mu^1_{0}&\cdots &\mu^{m-1}_{0}&1&0&\cdots &0
\end{matrix}\right|
$$
Proceeding in the same way $m$ times we have
$$
\Theta(x)=\left|\begin{matrix}
q^0_{m+n-1}&q^1_{m+n-1}&\cdots &q^{m-1}_{m+n-1}&0&0&\cdots &0\\
\vdots &\vdots&& \vdots &\vdots&\vdots& & \vdots \\
q^0_{n}&q^1_{n}&\cdots &q^{m-1}_{n}&0&0&\cdots &0\\
\mu^0_{n-1}&\mu^1_{n-1}&\cdots &\mu^{m-1}_{n-1}&r_{n-1}&r_{n-2}&\cdots &1\\
\vdots &\vdots&& \vdots &\vdots&\vdots& & \vdots \\
\mu^0_{0}&\mu^1_{0}&\cdots &\mu^{m-1}_{0}&1&0&\cdots &0
\end{matrix}\right| .
$$
And hence
$$
\Theta(x)=(-1)^{\binom{n}{2}}\left|\begin{matrix}
q^0_{m+n-1}&q^1_{m+n-1}&\cdots &q^{m-1}_{m+n-1}\\
\vdots &\vdots&& \vdots  \\
q^0_{n}&q^1_{n}&\cdots &q^{m-1}_{n}
\end{matrix}\right| .
$$
This finishes the proof of the Theorem.
\end{proof}

Theorem \ref{mti} in the Introduction is an easy corollary of Theorem \ref{mth} just taking into account the expression (\ref{fdpop}) for the monic orthogonal polynomials $\hat p_n$ with respect to a measure and that $p_n=\xi _n \hat p_n$.

\section{Specialization to some relevant operators}

For some relevant operators the estructure of the determinant in the right hand side of (\ref{mrby}) is richer. Indeed, as we can see in Leclerc
identity (\ref{leci}), for the derivative the polynomials $\psi_i$, $i\ge 0$, can be taken so that this determinant has a Hankel structure.  For the first order linear operator $\Delta $ there are also choices of the polynomials $\psi_i$, $i\ge 0$, so that this determinant has almost a Hankel structure. There also are other operators for which the polynomials $\psi_i$, $i\ge 0$, can be taken so that the determinant in the right hand side of (\ref{mrby}) has a Toeplitz structure. The operators defined by $T_\mu (p_n)=p_{n-1}$, where $(p_n)_n$ is a sequence of orthogonal polynomials with respect to any given measure $\mu$, are among these operators.

On the other hand, these relevant operators have associated distinguished sequences of orthogonal polynomials for which the polynomials $\psi_i$, $i\ge 0$, can be taken so that the polynomials $q_n^i$, $n,i\ge 0$, are again orthogonal. That is the case of the classical families of  Laguerre and Jacobi for the derivative and the classical discrete families of Charlier, Meixner, Krawtchouk and Hahn for the first order difference operator.
For these examples, the polynomials $q_n^i$ belong, up to a change of variable, to the same family (but with different parameters). This will allow us to prove our conjectures in \cite{du}. For certain measures $\mu$ (for instance, Laguerre, ultraspherical, Charlier and Meixner measures) and the corresponding operator $T_\mu$ (defined as above), the polynomials $\psi_i$, $i\ge 0$, can also be taken so that the polynomials $q_n^i$, $n,i\ge 0$, belong, up to a change of variable, to the same family (but with different parameters).

\subsection{The derivative}
In the case of the derivative, $T=d/dx$, we can take
\begin{align}\label{defrn1}
r_n(x)&=\frac{(x-x_0)^n}{n!},\\\label{defpsi1}
\psi_i(x)&=(x-x_0)^i,
\end{align}
where $x_0$ is a fixed number.
A simple computation then shows that
\begin{equation}\label{defsn1}
s_n(x)=\frac{(-x+x_0)^n}{n!},
\end{equation}
(see (\ref{defrn}) and (\ref{defsn}) for the definition of the polynomials $r_n, s_n$, $n\ge 0$.)
If we write $\mu_n=\int (x-x_0)^nd\mu$ for the generalized moments of the measure $\mu$, we then have
\begin{equation}\label{mijd1}
\mu^i_j=\int \psi_ir_jd\mu=\int \frac{(x-x_0)^{i+j}}{j!}d\mu=\frac{\mu_{i+j}}{j!}.
\end{equation}
From where it easily follows that $q_n^{i+1}-(x-x_0)q_n^i=(n+1)q_{n+1}^i$, and for $x_0=0$ $q_n=n!q^0_n$,  where $q_n$, $q_n^i$, $n,i\ge 0$, are
the polynomials defined by (\ref{defqnl}) and (\ref{defqny}), respectively. This shows that the determinant in the right hand side
of (\ref{mrby}) can be reduced to the Hankel determinant in the right hand side of Leclerc identity (\ref{leci}).

We now check that for the classical families of  Laguerre and Jacobi, the number $x_0$ and the sequence of polynomials  $(\psi_n)_n$ can be chosen such that the polynomials $q_n^i$ belong, up to a change of variable, to the same family (but with different parameters).

\subsubsection{Jacobi polynomials}\label{ss5.1}
For $\alpha,\beta \in \RR , \alpha+\beta \not=-1,-2,\cdots$, we use the standard definition of the Jacobi polynomials $(P_{n}^{\alpha,\beta})_n$ in terms of the hypergeometric function
\begin{equation}\label{defjac}
P_{n}^{\alpha,\beta}(x)=\binom{n+\alpha}{n}\sum _{j=0}^n \frac{2^{-j}(n+\alpha+\beta+1)_j}{(\alpha +1)_j}\binom{n}{j}(x-1)^{j}
\end{equation}
(that and the next formulas can be found in \cite{EMOT}, pp. 169-173).

For $\alpha,\beta, \alpha+\beta \not =-1,-2,\cdots$,
they are orthogonal with respect to a measure $\mu _{\alpha,\beta}=\mu _{\alpha,\beta}(x)dx$, which we
normalize by taking
$$
\int \mu _{\alpha,\beta}(x)dx=2^{\alpha+\beta+1}\frac{\Gamma(\alpha+1)\Gamma(\beta +1)}{\Gamma(\alpha+\beta+2)}.
$$
We write $c_{\alpha,\beta}$ for the constant in the right hand side of this identity.

The measure $\mu _{\alpha,\beta}$ is positive only when
$\alpha ,\beta >-1$, and then
\begin{equation}\label{jacw}
\mu_{\alpha,\beta}(x) =(1-x)^\alpha (1+x)^{\beta}, \quad -1<x<1.
\end{equation}
We now show that by choosing $x_0=1$, the polynomials $q_n^i$, $n,i\ge 0$, are, up to a change of variable, again Jacobi polynomials (but with different parameters). Indeed, we have that
$$
\mu_n^0=\int (x-1)^n\mu _{\alpha,\beta}(x)dx=c_{\alpha,\beta}\frac{(-2)^{n}(\alpha+1)_n}{(\alpha+\beta+2)_n}.
$$
Taking into account the definition of the polynomials $q_n^i$ (\ref{defqny}) and (\ref{mijd1}), we have
\begin{align}\nonumber
q_n^i(x)&=c_{\alpha,\beta}\sum_{j=0}^n\frac{(-2)^{i+j}(\alpha+1)_{i+j}}{(\alpha+\beta+2)_{i+j}j!}\frac{(1-x)^{n-j}}{(n-j)!}\\\label{qnij=j}
&=c_{\alpha,\beta}\frac{(-2)^i(\alpha+1)_i}{n!(\alpha+\beta+2)_i}\sum_{j=0}^n\frac{(-2)^{j}(\alpha+i+1)_{j}}{(\alpha+\beta+i+2)_{j}}\binom{n}{j}(1-x)^{n-j}.
\end{align}
The definition (\ref{defjac}) for the Jacobi polynomials gives
\begin{align}\nonumber
(1-x)^nP_{n}^{\alpha+\beta+i+1,-n-\beta-1}\left(\frac{x+3}{x-1}\right)&=\binom{n+\alpha+\beta+i+1}{n}\\\label{qnij=j2} &\hspace{1cm}\times\sum _{j=0}^n \frac{(-2)^{j}(\alpha+i+1)_j}{(\alpha + \beta+i+2)_j}\binom{n}{j}(1-x)^{n-j}.
\end{align}
Comparing (\ref{qnij=j}) and (\ref{qnij=j2}), we deduce that
$$
q_n^i(x)=\frac{(-2)^ic_{\alpha,\beta}(\alpha+1)_i}{(\alpha+\beta+2)_{n+i}}(1-x)^nP_{n}^{\alpha+\beta+i+1,-n-\beta-1}\left(\frac{x+3}{x-1}\right).
$$

\subsubsection{Laguerre polynomials}
For $\alpha\in \RR $, we use the standard definition of the Laguerre polynomials $(L_{n}^{\alpha})_n$
\begin{equation}\label{deflap}
L_{n}^{\alpha}(x)=\sum _{j=0}^n\binom{n+\alpha}{n-j}\frac{(-x)^{j}}{j!}
\end{equation}
(that and the next formulas can be found in \cite{EMOT}, pp. 188-192).

For $\alpha \not =-1,-2,\cdots $, they are  orthogonal with respect to a measure $\mu_\alpha=\mu_\alpha(x)dx$,
which we normalize by taking
$$
\int \mu _{\alpha}(x)dx=\Gamma(\alpha+1).
$$
The measure $\mu _{\alpha}$ is positive only when $\alpha >-1$,
and then
\begin{equation}\label{lagw}
\mu_\alpha (x) =x^\alpha e^{-x}, \quad 0<x.
\end{equation}
One can see that by choosing $x_0=0$, the polynomials $q_n^i$, $n,i\ge 0$, are, up to a change of variable, again Laguerre polynomials (but with different parameters). Indeed, proceeding as  before for the Jacobi polynomials, we deduce that
$$
q_n^i(x)= (-1)^n\Gamma(\alpha+i+1)L_{n}^{-\alpha-i-n-1}(-x).
$$

\subsection{The first order difference operator $\Delta$}
For the first order difference operator, $T=\Delta$,
we take
\begin{equation}\label{defrn2}
r_n(x)=\binom{x}{n},\quad \psi_i(x)=\binom{x}{i}.
\end{equation}
A simple computation then shows that
\begin{equation}\label{defsn2}
s_n(x)=\binom{-x}{n}.
\end{equation}
Using (\ref{abc2}), we get
\begin{align*}
\mu^i_j&=\int \psi_ir_jd\mu=\int \binom{x}{i}\binom{x}{j}d\mu=\sum_{l=0}^i\binom{j}{i-l}\int \binom{x-j}{l}\binom{x}{j}d\mu \\
&=\sum_{l=0}^i\binom{j}{i-l}\int \binom{j+l}{l}\binom{x}{j+l}d\mu
=\sum_{l=0}^i\binom{j}{i-l} \binom{j+l}{l}\mu ^0_{j+l}.
\end{align*}
Using now (\ref{abc1}), we find
\begin{align}\nonumber
q_n^i(x)&=\sum_{j=0}^n\mu^i_j\binom{-x}{n-j}=\sum_{j=0}^n\sum_{l=0}^i\binom{j}{i-l} \binom{j+l}{l}\mu ^0_{j+l}\binom{-x}{n-j}\\\nonumber
&=\sum_{g=i}^{n+i}\mu ^0_{g}\sum_{j=0}^{\min(n,g)}\binom{j}{i+j-g}\binom{g}{g-j}\binom{-x}{n-j}=\sum_{g=i}^{n+i}\mu ^0_{g}\binom{g}{i}\binom{-x+i}{n+i-g}\\\label{qniod}
&=\sum_{g=0}^{n}\mu ^0_{g+i}\binom{g+i}{i}\binom{-x+i}{n-g}.
\end{align}
This formula for $q_n^i$ allows us to show that the determinant in the right hand side of (\ref{mrby})
has almost a Hankel structure. Indeed, a simple computation using (\ref{qniod}) gives that
$$
q_n^{i+1}(x)+\frac{-x+i+1}{i+1}q_n^i(x)=\frac{n+1}{i+1}q_{n+1}^i (x-1).
$$
If we write $q_n=n!q^0_n$, it is easy to see that the determinant in the right hand side of (\ref{mrby}) reduces to
$$
\frac{1}{n!^m(1+i)^{m(m-1)/2}}\det\left(q_{n+i+j-2}(x-j+1)\right)_{i,j=1}^m.
$$

For the classical discrete families of Charlier, Meixner, Krawtchouk and Hahn, the polynomials
$q_n^i$, $n,i\ge 0$, are actually polynomials of the same family (but with different parameters) in the variable $-x+i$.
This will allow us to prove our conjectures in \cite{du} on Casorati determinants for Charlier, Meixner and Krawtchouk orthogonal polynomials.
Although we will consider real parameters, the identities (\ref{c1.n}), (\ref{me.n}), (\ref{kr.n}), (\ref{ha.n}), (\ref{dha.n}), (\ref{rac.n}) and (\ref{wil.n})
hold also for complex parameters (since both size of these identities are analytic functions of the parameters).

\subsubsection{Charlier polynomials} \label{ss5.2.1}
The monic Charlier polynomials are defined by
\begin{equation}\label{chpol}
\hat c_{n}^{a}(x)=\sum _{j=0}^n (-a)^{n-j}\binom{n}{j}\binom{x}{j}j!,\quad a\not =0,
\end{equation}
(that and the next formulas can be found in \cite{Ch}, pp. 170-1; see also \cite{KLS}, pp, 247-9 or \cite{NSU}, ch. 2).
The Charlier polynomials $(\hat  c_n^a)_n$ are orthogonal with respect to the discrete measure
\begin{equation}\label{chw}
\rho_a=\sum _{x\in \NN}\frac{a^x}{x!}\delta _x.
\end{equation}
The norm of the monic Charlier polynomials is given by
\begin{equation}\label{nmch}
\Vert \hat c_n^a\Vert ^2=e^an!a^n.
\end{equation}
We now show that the polynomials $q_n^i$, $n,i\ge 0$, are, up to a linear change of variable, again Charlier polynomials (but with parameter $-a$). Indeed, a simple computation gives $\mu^0_j=\int\binom{x}{j}d\rho_a=e^aa^j/j!$. Using (\ref{qniod}) we get
\begin{align}\nonumber
q_n^i(x)&=e^a\sum_{g=0}^{n}\frac{a^{g+i}}{(g+i)!}\binom{g+i}{i}\binom{-x+i}{n-g}\\\nonumber
&=\frac{e^aa^i}{i!n!}\sum_{g=0}^{n}a^g\binom{n}{n-g}\binom{-x+i}{n-g}(n-g)!\\\label{qnich}
&=\frac{e^aa^i}{i!n!}\hat c_n^{-a}(-x+i).
\end{align}
Consider now the following normalization of the Charlier polynomials
\begin{equation}\label{chno}
c_n^a(x)=\frac{\hat c_n^a(x)}{n!}, \quad n\ge 0,
\end{equation}
and the Casorati Charlier determinant
$$
C^a_{n,m,x}=\det\big(c_{m+j-1}^a(x+i-1)\big)_{i,j=1}^n
$$

\begin{corollary}
For $n,m\ge 0$ and $a\not = 0$, we have
\begin{equation}\label{c1.n}
C^a_{n,m,x}=(-1)^{nm}C^{-a}_{m,n,-x}.
\end{equation}
Moreover, conjecture 1.1 in \cite{du} (see (\ref{c1.1})) holds for $n,k,m\ge 0$ and $a\not = 0$.
\end{corollary}

\begin{proof}
The identity (\ref{c1.n}) follows straightforwardly from the identity (\ref{mrby}) for the first order difference operator $\Delta$, (\ref{qnich})
and the definition (\ref{chno}) of the polynomials $(c_n^a)_n$.

Using the duality
$$
\frac{(-1)^n}{a^n}\hat c_n^a(m)=\frac{(-1)^m}{a^m}\hat c_m^a(n),\quad n,m\ge 0,
$$
we can easily rewrite Conjecture 1.1 in \cite{du} (see (\ref{c1.1})) as the particular case of identity (\ref{c1.n}) when $x$ is a positive integer. This proves Conjecture 1.1 in \cite{du}.
\end{proof}

\subsubsection{Meixner polynomials}\label{ss5.2.2}
The monic Meixner polynomials are defined by
\begin{equation}\label{mpol}
\hat m_{n}^{a,c}(x)=\frac{(c)_n}{(1-1/a)^n}\sum _{j=0}^n \frac{(1-1/a)^j}{(c)_{j}}\binom{n}{j}\binom{x}{j}j!,\quad a\not =0, 1
\end{equation}
(we have taken a slightly different normalization from the one used in \cite{Ch}, pp. 175-7, from where
the next formulas can be easily derived; see also \cite{KLS}, pp, 234-7 or \cite{NSU}, ch. 2).
For $a\not =0,1$ and $c\not =0,-1,-2,\cdots $, they are always orthogonal with respect to a (possibly signed) measure $\rho_{a,c}$, which we
normalize by taking $\int d\rho_{a,c} =\Gamma(c)$. For $0<\vert a\vert<1$ and $c\not =0,-1,-2,\cdots $, we have
\begin{equation}\label{mew}
\rho_{a,c}=(1-a)^c\sum _{x=0}^\infty \frac{a^{x}\Gamma(x+c)}{x!}\delta _x.
\end{equation}
The moment functional $\rho_{a,c}$ can be represented by a positive measure only when $0<a<1$ and $c>0$.
The norm of the monic Meixner polynomials is given by
\begin{equation}\label{nmme}
\Vert \hat m_n^{a,c}\Vert ^2=\frac{a^n}{(1-a)^{2n}}n!\Gamma(c+n).
\end{equation}

A simple computation gives
$$
\mu^0_j=\frac{a^j\Gamma(c+j)}{(1-a)^{j}j!}.
$$
Proceeding as before for the Charlier polynomials, we can prove
that the polynomials $q_n^i$, $n,i\ge 0$, are, up to a linear change of variable, again Meixner polynomials (but with different parameters):
\begin{equation}\label{qnim}
q_n^i(x)=\frac{a^i\Gamma(c+i)}{i!n!(1-a)^i}\hat m_n^{a,-c-n-i+1}(-x+i).
\end{equation}
Consider now the following normalization of the Meixner polynomials
\begin{equation}\label{meno}
\textsl{\textsf{m}}_n^{a,c}(x)=\frac{\hat m_n^{a,c}(x)}{n!}, \quad n\ge 0,
\end{equation}
and the Casorati Meixner determinant
$$
M_{n,m,x}^{a,c}=\det\big(\textsl{\textsf{m}}_{m+j-1}^{a,c}(x+i-1)\big)_{i,j=1}^n.
$$
\begin{corollary}\label{meconjj}
For $n,m\ge 0$ and $a\not = 0, 1$, we have
\begin{equation}\label{me.n}
M_{n,m,x}^{a,c-n-m}=(-1)^{nm}M_{m,n,-x}^{a,2-c}.
\end{equation}
Moreover, Conjecture (3.7) in \cite{du}  holds for $n,k,m\ge 0$ and $a\not =0,1$.
\end{corollary}

\begin{proof}
Assume first $c-n-m\not =0,-1,-2,\cdots $, so that the polynomials $(\textsl{\textsf{m}}_n^{a,c-n-m})_n$ are orthogonal with respect to a measure.
The identity (\ref{me.n}) then follows from the identity (\ref{mrby}) for the first order difference operator $\Delta$, (\ref{qnim}), the definition (\ref{meno}) of the polynomials $(\textsl{\textsf{m}}_n^{a,c})_n$ and the identities
\begin{align*}
\textsl{\textsf{m}}_n^{a,c-1}(x)&=\textsl{\textsf{m}}_n^{a,c}(x)-\frac{a}{a-1}\textsl{\textsf{m}}_{n-1}^{a,c}(x),\\
\textsl{\textsf{m}}_n^{a,c+1}(x)&=\frac{a}{a-1}\textsl{\textsf{m}}_n^{a,c}(x+1)-\frac{1}{a-1}\textsl{\textsf{m}}_{n}^{a,c}(x).
\end{align*}
Since both sides of the identity (\ref{me.n}) are polynomials in $c$, this identity also holds for $c-n-m=0,-1,-2,\cdots $.

Using the duality
$$
\frac{(a-1)^n}{a^n(1+c)_{n-1}}\hat m_n^{a,c}(m)=\frac{(a-1)^m}{a^m(1+c)_{m-1}}\hat m_m^{a,c}(n),\quad n,m\ge 0,
$$
we can easily rewrite Conjecture (3.7) in \cite{du} as the particular case of identity (\ref{me.n}) when $x$ is a positive integer. This proves Conjecture (3.7) in \cite{du}.
\end{proof}

\subsection{Krawtchouk polynomials}
For $a\not=0,-1$, we write $(\hat k_{n}^{a,N})_n$ for the sequence of monic Krawtchouk polynomials defined by
\begin{equation}\label{krpol}
\hat k_{n}^{a,N}(x)=\frac{(-1)^{n}a^{n}}{(1+a)^{n}}\sum_{j=0}^n(-1)^{j}\frac{(1+a)^{j}(N-n)_{n-j}}{a^{j}}\binom{n}{j}\binom{x}{j}.
\end{equation}
For $a\not=0,-1$ and $N\not=1, 2, \cdots$, they are always orthogonal with respect to a (signed) measure $\rho _{a,N}$ which we normalize by taking $\langle \rho_{a,N},1\rangle=1$.
When $N$ is a positive integer and $a>0$, the first $N$ polynomials are orthogonal with respect to the positive Krawtchouk measure
\begin{equation}\label{krw}
\rho_{a,N}=\frac{\Gamma(N)}{(1+a)^{N-1}}\sum _{x=0}^{N-1} \frac{a^{x}}{\Gamma(N-x)x!}\delta _x.
\end{equation}
A simple computation gives
$$
\mu^0_j=\frac{(N-j)_ja^j}{(1+a)^{j}j!}.
$$
Proceeding as before for the Charlier polynomials, we can prove
that the polynomials $q_n^i$, $n,i\ge 0$, are, up to a linear change of variable, again Krawtchouk polynomials (but with different parameters):
\begin{equation}\label{qnik}
q_n^i(x)=\frac{a^i(N-i)_i}{i!n!(1+a)^i}\hat k_n^{a,-N+i+n+1}(-x+i).
\end{equation}

Consider now the following normalization of the Krawtchouk polynomials
\begin{equation}\label{krno}
k_n^{a,N}(x)=\frac{\hat k_n^{a,N}(x)}{n!}, \quad n\ge 0,
\end{equation}
and the Casorati Krawtchouk determinant
$$
K_{n,m,x}^{a,N}=\det\big(k_{m+j-1}^{a,N}(x+i-1)\big)_{i,j=1}^n.
$$
\begin{corollary}
For $n,m\ge 0$ and $a\not =0, -1$, we have
\begin{equation}\label{kr.n}
K_{n,m,x}^{a,N+n+m}=(-1)^{nm}K_{m,n,-x}^{a,-N}.
\end{equation}
Moreover, Conjecture (5.3) in \cite{du}  holds for $n,k,m\ge 0$ and $a\not =0,-1$.
\end{corollary}

\begin{proof}
We can proceed as in the proof of Corollary \ref{meconjj} but using now
the identities
\begin{align*}
k_n^{a,N+1}(x)&=k_n^{a,N}(x)-\frac{a}{1+a}k_{n-1}^{a,N}(x),\\
k_n^{a,N-1}(x)&=\frac{a}{1+a}k_n^{a,N}(x+1)+\frac{1}{1+a}k_{n}^{a,N}(x),
\end{align*}
and the duality
$$
\frac{n!(a+1)^n}{a^n(2-N)_{n-1}}k_n^{a,N}(m)=\frac{m!(a+1)^m}{a^m(2-N)_{m-1}}k_m^{a,N}(n),\quad n,m\ge 0.
$$
\end{proof}

\subsubsection{Hahn polynomials}
For $\alpha+c-N\not =-1, -2, \cdots$ we write $(\hat h_{n}^{\alpha ,c,N})_n$ for the sequence of monic Hahn polynomials defined by
\begin{equation}\label{defhap}
\hat h_{n}^{\alpha ,c,N}(x)=\sum_{j=0}^n\frac{(1-N+j)_{n-j}(c+j)_{n-j}}{(n+\alpha +c-N+j)_{n-j}}\binom{n}{j}\binom{x}{j}j!
\end{equation}
(we have taken a slightly different normalization from the one used in \cite{NSU}, pp. 30-53, from where
the next formulas can be easily derived; for more details see \cite{KLS}, pp, 204-211).

Assume that $\alpha+c-N+1, \alpha-N+1, \alpha +c, c\not =0,-1, -2, \cdots $.
If, in addition, $N$ is not a positive integer, then the Hahn polynomials are always orthogonal with respect to
a (possibly signed) measure $\rho_{\alpha,c,N}$, which we normalize by taking
$$
\int d\rho_{\alpha,c,N}=\frac{\Gamma(\alpha+1-N)\Gamma(c)\Gamma(\alpha+c)}{\Gamma(\alpha+c+1-N)}.
$$
When $N$ is a positive integer, the first $N$ Hahn polynomials are orthogonal
with respect to the  Hahn measure
\begin{equation}\label{haw}
\rho_{\alpha,c,N}=\Gamma(N)\sum _{x=0}^{N-1} \frac{\Gamma(\alpha-x)\Gamma(x+c)}{\Gamma(N-x)x!}\delta _x
\end{equation}
(which it is positive when $\alpha>N-1$ and $c>0$).
The norm of the monic Hahn polynomials is given by
\begin{equation}\label{nmh}
\Vert \hat h_n^{\alpha , c,N}\Vert ^2=\frac{n!(N-n)_n\Gamma(\alpha-N+n+1)\Gamma(c+n)\Gamma(c+\alpha+n)}{(\alpha-N+c+2n)\Gamma(\alpha-N+c+n)(\alpha-N+c+n)_n^2}.
\end{equation}

A computation gives
$$
\mu^0_j=\frac{(N-j)_j\Gamma(\alpha-N+1)\Gamma(c+j)\Gamma(c+\alpha )}{\Gamma(c+\alpha-N+j+1)j!}.
$$
Proceeding as before for the Charlier polynomials, we can prove
that the polynomials $q_n^i$, $n,i\ge 0$, are, up to a linear change of variable, again Hahn polynomials (but with different parameters):
$$
q_n^i(x)=\frac{\mu_i^0}{n!}\hat h_n^{i-\alpha,-c-n-i+1,n+1+i-N}(-x+i).
$$
Consider now the following normalization of the Hahn polynomials
\begin{equation}\label{hano}
h_n^{\alpha ,c,N}(x)=\frac{\hat h_n^{\alpha ,c,N}(x)}{n!}, \quad n\ge 0,
\end{equation}
and the Casorati Hahn determinant
$$
H_{n,m,x}^{\alpha ,c,N}=\det\big(h_{m+j-1}^{\alpha ,c,N}(x+i-1)\big)_{i,j=1}^n.
$$
Proceeding in a similar way to the previous Sections, we can prove the following result.
\begin{corollary}
For $n,m\ge 0$ and  $\alpha+c-N\not =-1, -2, \cdots$, we have
\begin{equation}\label{ha.n}
H_{n,m,x}^{\alpha ,c-n-m,N+n+m}=(-1)^{nm}H_{m,n,-x}^{-\alpha ,2-c,-N}.
\end{equation}
\end{corollary}

We stress that this symmetry for the Casorati determinants of Hahn polynomials is essentially different to the ones pose in the Conjectures 5.1 in \cite{du}.

\subsection{A composition of $\Delta$ and a polynomial of second degree}
In order to prove the Conjectures 5.1 in \cite{du}, we need to modify Theorems \ref{mth} and \ref{mti} for a composition of the first order difference operator $\Delta$ and a second degree polynomial.

For a complex number $u$, write
\begin{equation}\label{deflam}
\lambda(x)=x(x+u),
\end{equation}
and consider the polynomials $f_j$, $r_j$ and $s_j$, $j\ge 0$, defined by
\begin{align}
\label{deffj} f_j(x)&=\frac{(-1)^j}{j!}\prod_{i=0}^{j-1}(x-i(u+i)),\\
\label{defrj} r_j(x)&=\frac{\displaystyle (-x)_j(x+u)_j}{j!},\\
\label{defsj} s_j(x)&=(x)_j\binom{x+u+n+m-2}{j}.
\end{align}
Let us notice that $f_j$, $r_j$, $j\ge 0$, are polynomials of degree $j$ and $2j$, respectively, depending on $j$ and $u$, whereas $s_j$ is a polynomial of degree $2j$ depending on $j, u, m$ and $n$. When we need to stress those dependencies we will write $r_j^u$, $s_j^{u,m,n}$ or just $s_j^n$.

The following identities are simple exercises.

\begin{lemma}\label{lem5.4}
For $k,l\ge 0$, $0\le l\le k$,
\begin{align}
\label{5.4.1} r_k(x)&=f_k(\lambda(x)),\\
\label{5.4.2}
\Delta^lr^u_k(x)&=(-1)^l \sum_{j=0}^{[l/2]}\frac{(-1)^j}{j!}(l+1-2j)_{2j}(2x+u+l)_{l-2j}r^{u+l-j}_{k-l+j}(x).
\end{align}
For $0\le l\le n$ and $k\ge 0$
\begin{align}
\label{5.4.3}\nabla s_{k+1}^{n-l+1}(x)=(2x+u+m+n-l-2)s_{k}^{n-l}(x),
\end{align}
where $\nabla$ is the first order difference operator defined by $\nabla f=f(x)-f(x-1)$.
For $l\ge 0$, $k\in \ZZ $ and $m\ge \max\{0,-k+1\}$, we have
\begin{align}
\label{5.4.4}
\sum_{j=0}^{m+k}s^{n+k-l}_j(x)r_{m+k-j}^{u+n-1-l}(x)=0.
\end{align}

\end{lemma}

We then have the following modified version of Theorems \ref{mth} and \ref{mti}.

\begin{theorem}\label{mth2} For a complex number $u$, consider the polynomial $\lambda(x)=x(x+u)$ and the first order difference operator $\Delta$.
Let $\langle ,\rangle$ be a bilinear form acting in the linear space of polynomials.
For a given sequence of polynomials $(\psi_i)_i$, $\psi _i$ of degree $i$, we write $\mu_j^i=\langle f_j,\psi_i\rangle$, $0\le i,j$,
where $f_j$ is the polynomial defined by (\ref{deffj}).
Given two nonnegative integers $n,m$, with the operator $\Delta $, the polynomial $\lambda$ and the bilinear form $\langle ,\rangle$ we associate the
polynomials $p_k, q_k^i$, $k,i\ge0$. The polynomials $p_k$, $k\ge0$, are the pseudo orthogonal polynomials with respect to
$\langle ,\rangle$ defined by $p_0=1$ and for $k\ge 1$
\begin{equation}\label{fdpo3}
p_k(x)=\left| \begin{matrix} \mu _k^0 & \mu _k^1&\cdots & \mu _k^{k-1}&f_k(x)\\
\mu _{k-1}^0 &\mu_{k-1}^1& \cdots &\mu_{k-1}^{k-1}&f_{k-1}(x)\\ \vdots &\vdots & \ddots & \vdots& \vdots
\\\mu_0^0 & \mu_0^1&\cdots &\mu_0^{k-1}&1 \end{matrix}\right| .
\end{equation}
The polynomials $q_k^i$, $k,i\ge0$, are defined by
\begin{equation}\label{defqn2}
q_k^i(x)=\sum_{j=0}^k\mu _j^is_{k-j}^{u,m,n}(x),
\end{equation}
where $s_k^{u,m,n}$ are the polynomials defined by (\ref{defsj}).
We then have
\begin{equation}\label{mrb2}
\det\left(\Delta ^{i-1}(p_{m+j-1}(\lambda(x))\right)_{i,j=1}^n=C_{n,m}\det\left( q^{j-1}_{n+m-1}(x-i+1)\right)_{i,j=1}^m,
\end{equation}
where
$$
C_{n,m}=\frac{(-1)^{\binom{m}{2}}\prod_{j=1}^{n-1}(2x+u+j)_j}{\prod_{j=1}^{m-1}(2x+u+j+n-m)_j}\prod_{j=0}^{n-2}\Omega_{m+j}
$$
(and $\Omega_n$ is defined by (\ref{defome})).
Moreover, if $\mu$ is a measure, $\langle p,q\rangle =\int pqd\mu$ and $(p_n)_n$ is a sequence of orthogonal polynomials with respect to $\mu$ (see Section 2 for more details) then
\begin{equation}\label{mrby2}
\Omega_{m-1}\det \left( \Delta ^{i-1}(p_{m+j-1}(\lambda(x)))\right)_{i,j=1}^n=D_{n,m}\det \left( q^{j-1}_{n+m-1}(x-i+1)\right)_{i,j=1}^m ,
\end{equation}
where
$$
D_{n,m}=m!\xi_m\frac{(-1)^{\binom{m}{2}}\prod_{j=1}^{n-1}(m+j)!\xi_{m+j}(2x+u+j)_j}{\prod_{j=1}^{m-1}(2x+u+j+n-m)_j}
$$
and $\xi_n$ is the coefficient of $x^n$ in the power expansion of $p_n$.

\end{theorem}

\begin{proof}
Using (\ref{5.4.1}), we get that
$$
p_k(\lambda(x))=\left| \begin{matrix} \mu _k^0 & \mu _k^1&\cdots & \mu _k^{k-1}&r_k(x)\\
\mu _{k-1}^0 &\mu_{k-1}^1& \cdots &\mu_{k-1}^{k-1}&r_{k-1}(x)\\ \vdots &\vdots & \ddots & \vdots& \vdots
\\\mu_0^0 & \mu_0^1&\cdots &\mu_0^{k-1}&1 \end{matrix}\right| .
$$
Hence, proceeding as in the proof of Theorem \ref{mth}, we get that the determinant in the left hand side of (\ref{mrb2}) is equal to
$$
\Omega_{m+n-2}\Omega_{m+n-3}\cdots \Omega _m\Theta(x)
$$
where
$$
\Theta(x)=\left|\begin{matrix}
\mu^0_{m+n-1}&\mu^1_{m+n-1}&\cdots &\mu^{m-1}_{m+n-1}&r_{m+n-1}&\Delta (r_{m+n-1})&\cdots &\Delta^{n-1}(r_{m+n-1})\\
\mu^0_{m+n-2}&\mu^1_{m+n-2}&\cdots &\mu^{m-1}_{m+n-2}&r_{m+n-2}&\Delta (r_{m+n-2})&\cdots &\Delta^{n-1}(r_{m+n-2})\\
\vdots &\vdots&& \vdots &\vdots&\vdots& & \vdots \\
\mu^0_{1}&\mu^1_{1}&\cdots &\mu^{m-1}_{1}&r_{1}&\Delta(r_{1})&\cdots &0\\
\mu^0_{0}&\mu^1_{0}&\cdots &\mu^{m-1}_{0}&1&0&\cdots &0
\end{matrix}\right|,
$$
and $\Omega _n$ is defined by (\ref{defome}).

Using (\ref{5.4.2}), we get after combining the last $n$ columns of this determinant that
$$
\frac{(-1)^{\binom{n}{2}}\Theta (x)}{\prod_{j=1}^{n-1}(2x+u+j)_j}= \left|\begin{matrix}
\mu^0_{m+n-1}&\mu^1_{m+n-1}&\cdots &\mu^{m-1}_{m+n-1}&r_{m+n-1}^u&r_{m+n-2}^{u+1}&\cdots &r_{m}^{u+n-1}\\
\mu^0_{m+n-2}&\mu^1_{m+n-2}&\cdots &\mu^{m-1}_{m+n-2}&r_{m+n-2}^u&r_{m+n-3}^{u+1}&\cdots &r_{m-1}^{u+n-1}\\
\vdots &\vdots&& \vdots &\vdots&\vdots& & \vdots \\
\mu^0_{1}&\mu^1_{1}&\cdots &\mu^{m-1}_{1}&r_{1}^u&1&\cdots &0\\
\mu^0_{0}&\mu^1_{0}&\cdots &\mu^{m-1}_{0}&1&0&\cdots &0
\end{matrix}\right|.
$$
We now consider the polynomials defined by
$$
q_k^{i,l}(x)=\sum_{j=0}^k\mu_j^is_{k-j}^{u,m,n-l}(x).
$$
Notice that $q_k^{i,0}=q_k^i$, where $q_k^i$ are the polynomials defined by (\ref{defqn2}).

We then add to the first row of the previous determinant
the linear combination of the next $m+n-1$ rows
multiplying the $j$-th row by $s_{j-1}^{u,m,n}(x)$, $j=2,\cdots, m+n$. Using (\ref{5.4.4}) for $k=l$ and $l=0,\cdots ,n-1$ we get
$$
\frac{(-1)^{\binom{n}{2}}\Theta (x)}{\prod_{j=1}^{n-1}(2x+u+j)_j}=\left|\begin{matrix}
q^{0,0}_{m+n-1}&q^{1,0}_{m+n-1}&\cdots &q^{m-1,0}_{m+n-1}&0&0&\cdots &0\\
\mu^0_{m+n-2}&\mu^1_{m+n-2}&\cdots &\mu^{m-1}_{m+n-2}&r_{m+n-2}^u&r_{m+n-3}^{u+1}&\cdots &r_{m-1}^{u+n-1}\\
\vdots &\vdots&& \vdots &\vdots&\vdots& & \vdots \\
\mu^0_{1}&\mu^1_{1}&\cdots &\mu^{m-1}_{1}&r_{1}&1&\cdots &0\\
\mu^0_{0}&\mu^1_{0}&\cdots &\mu^{m-1}_{0}&1&0&\cdots &0
\end{matrix}\right|.
$$
In a similar way, we add to the second row of this determinant the linear combination of the next $m+n-2$ rows
multiplying the $j$-th row by $s_{j-2}^{u,m,n-1}(x)$, $j=3,\cdots, m+n$. Using (\ref{5.4.4}) for $k=l-1$ and $l=0,\cdots ,n-1$ we get
$$
\frac{(-1)^{\binom{n}{2}}\Theta (x)}{\prod_{j=1}^{n-1}(2x+u+j)_j}=\left|\begin{matrix}
q^{0,0}_{m+n-1}&q^{1,0}_{m+n-1}&\cdots &q^{m-1,0}_{m+n-1}&0&0&\cdots &0\\
q^{0,1}_{m+n-2}&q^{1,1}_{m+n-2}&\cdots &q^{m-1,2}_{m+n-2}&0&0&\cdots &0\\
\mu^0_{m+n-3}&\mu^1_{m+n-3}&\cdots &\mu^{m-1}_{m+n-3}&r_{m+n-3}^u&r_{m+n-4}^{u+1}&\cdots &r_{m-2}^{u+n-1}\\
\vdots &\vdots&& \vdots &\vdots&\vdots& & \vdots \\
\mu^0_{1}&\mu^1_{1}&\cdots &\mu^{m-1}_{1}&r_{1}^u&1&\cdots &0\\
\mu^0_{0}&\mu^1_{0}&\cdots &\mu^{m-1}_{0}&1&0&\cdots &0
\end{matrix}\right| .
$$
Proceeding in the same way $m$ times we have
$$
\frac{(-1)^{\binom{n}{2}}\Theta (x)}{\prod_{j=1}^{n-1}(2x+u+j)_j}=\left|\begin{matrix}
q^{0,0}_{m+n-1}&q^{1,0}_{m+n-1}&\cdots &q^{m-1,0}_{m+n-1}&0&0&\cdots &0\\
\vdots &\vdots&& \vdots &\vdots&\vdots& & \vdots \\
q^{0,m-1}_{n}&q^{1,m-1}_{n}&\cdots &q^{m-1,m-1}_{n}&0&0&\cdots &0\\
\mu^0_{n-1}&\mu^1_{n-1}&\cdots &\mu^{m-1}_{n-1}&r_{n-1}&r_{n-2}&\cdots &1\\
\vdots &\vdots&& \vdots &\vdots&\vdots& & \vdots \\
\mu^0_{0}&\mu^1_{0}&\cdots &\mu^{m-1}_{0}&1&0&\cdots &0
\end{matrix}\right| .
$$
And hence
$$
\Theta(x)=\prod_{j=1}^{n-1}(2x+u+j)_j\left|\begin{matrix}
q^{0,0}_{m+n-1}&q^{1,0}_{m+n-1}&\cdots &q^{m-1,0}_{m+n-1}\\
\vdots &\vdots&& \vdots  \\
q^{0,m-1}_{n}&q^{1,m-1}_{n}&\cdots &q^{m-1,m-1}_{n}
\end{matrix}\right| .
$$
Using  (\ref{5.4.3}), we easily have that
$$
q_{k+1}^{j,l-1}(x)-q_{k+1}^{j,l-1}(x-1)=(2x+u+m+n-l-2)q_k^{j,l}.
$$
Taking into account that $q_k^{i,0}=q_k^i$, where $q_k^i$ are the polynomials defined by (\ref{defqn2}), by combining rows of the previous determinant we get
$$
\Theta(x)=E_{n,m}\left|\begin{matrix}
q^{0}_{m+n-1}(x)&q^{1}_{m+n-1}(x)&\cdots &q^{m-1}_{m+n-1}(x)\\
\vdots &\vdots&& \vdots  \\
q^{0}_{m+n-1}(x-m+1)&q^{1}_{m+n-1}(x-m+1)&\cdots &q^{m-1}_{m+n-1}(x-m+1)
\end{matrix}\right| ,
$$
where
$$
E_{n,m}=\frac{(-1)^{\binom{m}{2}}\prod_{j=1}^{n-1}(2x+u+j)_j}{\prod_{j=1}^{m-1}(2x+u+j+n-m)_j}.
$$
This finishes the proof of the identity (\ref{mrb2}).

The identity (\ref{mrby2}) is an easy consequence of (\ref{mrb2}) just taking into account the expression (\ref{fdpop}) for the monic orthogonal polynomials $\hat p_n$ with respect to a measure and that $p_n=\xi _n \hat p_n$.
\end{proof}

Given a measure $\mu$, if we take $\psi_j(x)=f_j(x)$ (see (\ref{deffj})), an expression of the polynomials $q_{n+m-1}^i$ in terms of the generalized moments $\mu_n^0=\int f_n(x)d\mu (x)$ of $\mu $ can be given.

Indeed, consider the polynomials $f_{k,m}(x)$, $k,m\ge 0$, defined by
$$
f_{k,m}(x)=\frac{(-1)^m}{m!}\prod _{i=k}^{k+m-1}(x-i(u+i)).
$$
Notice that $f_j=f_{0,j}$, $j\ge 0$.

The following identities are simple exercises.

\begin{lemma} For $i,j,l\ge 0$, $0\le l\le i$,
\begin{align}\label{fj1}
f_{j,i-l}(x)f_j(x)&=\binom{j+i-l}{j}f_{j+i-l}(x),\\ \label{fj2}
f_i(x)&=\sum_{l=0}^{\min (i,j)}(-1)^l\binom{j}{l}(u+i+j-l)_lf_{j,i-l}(x).
\end{align}
\end{lemma}

We then have
\begin{align*}
\mu_j^i&=\int \psi_i f_jd\mu=\int f_if_jd\mu \\
&=\sum_{l=0}^{\min (i,j)}(-1)^l\binom{j}{l}(u+i+j-l)_l\int f_if_{j,i-l}d\mu \\
&=\sum_{l=0}^{\min (i,j)}(-1)^l\binom{j}{l}\binom{j+i-l}{j}(u+i+j-l)_l\int f_{i+j-l}d\mu \\
&=\sum_{l=0}^{\min (i,j)}(-1)^l\binom{j}{l}\binom{j+i-l}{j}(u+i+j-l)_l\mu_{i+j-l}^0.
\end{align*}
This gives
\begin{align*}
q_k^i(x)&=\sum_{j=0}^k\mu _j^is_{k-j}(x)\\
&=\sum_{j=0}^k\sum_{l=0}^{\min (i,j)}(-1)^l\binom{j}{l}\binom{j+i-l}{j}(u+i+j-l)_l\mu_{i+j-l}^0(x)_{k-j}\binom{x+u+n+m-2}{k-j}\\
&=\sum_{g=0}^{k}\mu_{g+i}^0\sum_{j=0}^k(-1)^{j-g}\binom{j}{j-g}\binom{g+i}{j}(u+g+i)_{j-g}(x)_{k-j}\binom{x+u+n+m-2}{k-j}.
\end{align*}

Using (\ref{abc3}), we finally find
\begin{align*}
q_{n+m-1}^i(x)=\sum_{g=0}^{n+m-1}\mu_{g+i}^0\binom{g+i}{i}(x-i)_{n+m-1-g}\frac{(x+u+g+i)_{n+m-1-g}}{(n+m-1-g)!}.
\end{align*}

\bigskip

\subsubsection{Dual Hahn polynomials}
For $\alpha, c, N\in \RR $ we write $(\hat h_{n}^{*,\alpha,c,N})_n$ for the sequence of monic dual Hahn polynomials (see \cite{KLS}, pp, 208-213)
defined by
\begin{equation}\label{defdualhap}
\hat h_{n}^{*,\alpha,c,N}(x)=\sum_{j=0}^n(-n)_j(1-N+j)_{n-j}(c+j)_{n-j}f_{j}(x),
\end{equation}
where
$$
f_j(x)=\frac{(-1)^j}{j!}\prod_{i=0}^{j-1}(x-i(\alpha+c-N+i)),
$$
(that is, the polynomials defined by (\ref{deffj}) for $u=\alpha+c-N$).

With the family $(h_{n}^{*,\alpha ,c,N})_n$ we associate the second degree polynomial $\lambda(x)=\lambda^{\alpha ,c,N}(x)=x(x+\alpha +c-N)$.

Assume that $\alpha,  -c\not =0, 1, 2, \cdots $.
If, in addition, $N$ is not a positive integer, then the dual Hahn polynomials are always orthogonal with respect to
a (possibly signed) measure $\rho_{\alpha,c,N}$, which we normalize by taking
$$
\int d\rho_{\alpha,c,N}=\frac{\Gamma(\alpha+1-N)}{\Gamma(\alpha)}.
$$
When $N$ is a positive integer, the first $N$ dual Hahn polynomials are orthogonal
with respect to the dual Hahn measure
\begin{equation}\label{dhaw}
\rho_{\alpha,c,N}=\sum _{x=0}^{N-1} \frac{(2x+\alpha+c-N)(c)_x(N-x)_x}{(x+\alpha+c-N)_N(\alpha-N+1)_xx!}\delta _{\lambda (x)}
\end{equation}
(which it is positive when $\alpha>N-1$ and $c>0$).
The norm of the monic dual Hahn polynomials is given by
\begin{equation}\label{nmdh}
\Vert \hat h_n^{*,\alpha , c,N}\Vert ^2=\frac{n!(c)_n\Gamma(\alpha -N+1)\Gamma(N)}{\Gamma(a-n)\Gamma(N-n)}.
\end{equation}

A computation gives
$$
\mu^0_j=(-1)^j\frac{\Gamma(\alpha-N+1)\Gamma(N)(c)_j}{j!\Gamma(\alpha)\Gamma(N-j)}.
$$
Proceeding as before for the Charlier polynomials, we can prove
that the polynomials $q_{n+m-1}^i$, $0\le m-1$, are, up to a quadratic change of variable, again dual Hahn polynomials:
\begin{equation}\label{dhqni}
q_{n+m-1}^i(x)=\frac{\mu_i^0}{(n+m-1)!}
\hat h_{n+m-1}^{*,-\alpha+n+m,2-c-n-m-i,-N+n+m+i}(\lambda(-x+i)).
\end{equation}
Consider now the following normalization of the dual Hahn polynomials
\begin{equation}\label{dhno}
h_n^{*,\alpha,c,N}(x)=\frac{\hat h_n^{*,\alpha,c,N}(x)}{n!}, \quad n\ge 0,
\end{equation}
and the (normalized) Casorati dual Hahn determinant
$$
H_{n,m,x}^{*,\alpha,c,N}=\frac{\det\big( h_{m+j-1}^{*,\alpha,c,N}(\lambda^{\alpha,c,N}(x+i-1))\big)_{i,j=1}^n}
{\prod_{l=1}^{n-1}(2x+\alpha+c-N+l)_l}.
$$
We are now ready to prove the Conjecture (5.8) in \cite{du}.

\begin{corollary}
For $n,m\ge 0$, we have
\begin{equation}\label{dha.n}
H_{n,m,x}^{*,\alpha+n+m,c-n-m,N+n+m}=H_{m,n,-x}^{*,-\alpha,2-c,-N}.
\end{equation}
Moreover, Conjecture (5.8) in \cite{du}  holds for $n,k,m\ge 0$ and $\alpha+c-N\not =0,-1,-2\cdots $.
\end{corollary}

\begin{proof}
Assume first $\alpha, -c\not =0,1,2,\cdots $, so that the polynomials $(h_n^{*,\alpha ,c,N})_n$ are orthogonal with respect to a measure.
The identity (\ref{dha.n}) then follows from the identity (\ref{mrby2}), (\ref{dhqni}), the definition (\ref{dhno}) of the polynomials $(h_n^{*,\alpha ,c,N})_n$ and the identity
\begin{align*}
h_k^{*,\alpha,c,N}(\lambda ^{\alpha , c, N}(x))&=h_k^{*,\alpha,c-1,N+1}(\lambda ^{\alpha , c-1, N+1}(x+1))\\
&\hspace{1cm}-(k-\alpha )h_{k-1}^{*,\alpha,c,N}(\lambda ^{\alpha , c, N}(x)).
\end{align*}
Since both sides of the identity (\ref{dha.n}) are rational functions in $c$ and $\alpha $, this identity also holds for $\alpha, -c =0,1,2,\cdots $.

Using the duality
$$
\frac{1}{(c)_k(1-N)_k}\hat h_k^{*,\alpha,c,N}(n(n+\alpha+c-N))=\frac{(n+\alpha +c-N)_n}{(c)_n(1-N)_n}\hat h_n^{\alpha,c,N}(k).
$$
we can  rewrite Conjecture (5.8) in \cite{du} as the particular case of identity (\ref{dha.n}) when $x$ is a positive integer. This proves Conjecture (5.8) in \cite{du}.
\end{proof}

\subsubsection{Racah and Wilson polynomials}\label{5.4.2rw}
For the sake of completeness, we include here without details
some symmetries for determinants of Racah and Wilson polynomials.

For $\alpha, \beta, \gamma, \delta\in \RR $, $\alpha+\beta\not=-2,-3,\cdots$, we write $(r_{n}^{\alpha,\beta,\gamma,\delta})_n$ for the sequence of Racah polynomials
defined by
$$
r_{n}^{\alpha,\beta,\gamma,\delta}(x)=\sum_{j=0}^n\frac{
(\alpha+1+j)_{n-j}(\beta+\delta+1+j)_{n-j}(\gamma+1+j)_{n-j}}{(-1)^j(n-j)!(n+\alpha+\beta+1+j)_{n-j}}f_{j}(x),
$$
(we have taken a slightly different normalization from the one used in \cite{KLS}, pp, 190-196) where
$$
f_j(x)=\frac{(-1)^j}{j!}\prod_{i=0}^{j-1}(x-i(\gamma+\delta+1+i)),
$$
(that is, the polynomials defined by (\ref{deffj}) for $u=\gamma+\delta+1$).

Consider now the following (normalized) Casorati Racah determinant
$$
R_{n,m,x}^{\alpha,\beta,\gamma,\delta}=\frac{\det\big( r_{m+j-1}^{\alpha,\beta,\gamma,\delta}(\lambda^{\gamma,\delta}(x+i-1))\big)_{i,j=1}^n}
{\prod_{l=1}^{n-1}(2x+\gamma+\delta+1+l)_l},
$$
where $\lambda^{\gamma,\delta}(x)=x(x+\gamma+\delta+1)$.

\begin{corollary}
For $n,m\ge 0$, we have
\begin{equation}\label{rac.n}
R_{n,m,x}^{\alpha-n-m,\beta-n-m,\gamma-n-m,\delta}=R_{m,n,-x}^{-\alpha,-\beta,-\gamma,-\delta}.
\end{equation}
\end{corollary}

\bigskip

For $a, b, c, d\in \RR $, $a+b+c+d\not =0,-1,-2,\cdots$, we write $(w_{n}^{a,b,c,d})_n$ for the sequence of Wilson polynomials
 defined by
$$
w_{n}^{\alpha,\beta,\gamma,\delta}(x)=\sum_{j=0}^n\frac{
(a+b+j)_{n-j}(a+c+j)_{n-j}(a+d+j)_{n-j}}{(n-j)!(n+a+b+c+d-1+j)_{n-j}}g_{j}(x),
$$
where
$$
g_j(x)=\frac{(-1)^j}{j!}\prod_{i=0}^{j-1}(x+(a+i)^2)
$$
(we have taken a slightly different normalization from the one used in  \cite{KLS}, pp, 185-190).

Consider now the following (normalized) Casorati Wilson determinant
$$
W_{n,m,x}^{a,b,c,d}=\frac{\det\big( w_{m+j-1}^{a,b,c,d}(-(x+i-1)^2)\big)_{i,j=1}^n}
{\prod_{l=1}^{n-1}(2x+l)_l}.
$$

\begin{corollary}
For $n,m\ge 0$, we have
\begin{equation}\label{wil.n}
W_{n,m,x+a-\frac{n}{2}-\frac{m}{2}}^{a-\frac{n}{2}-\frac{m}{2},b-\frac{n}{2}-\frac{m}{2},c-\frac{n}{2}-\frac{m}{2},d-\frac{n}{2}-\frac{m}{2}}=
W_{m,n,-x-a+1}^{-d+1,-b+1,-a+1,-c+1}.
\end{equation}
\end{corollary}

\subsection{The operator $T_\mu (p_n)=p_{n-1}$}\label{ss5.3}
Let $(p_n)_n$ be a sequence of orthogonal polynomials with respect to a positive measure $\mu$, and consider the linear operator
$T_\mu$ defined by linearity from $T_\mu(p_n)=p_{n-1}$. Notice that the determinant in the left hand side of (\ref{mrby}) has now a Toeplitz structure.

We take
$$
r_n(x)=p_n(x),\quad \psi_i(x)=p_i(x).
$$
Define finally the polynomials $s_n$, $n\ge 0$, by (\ref{defsn}).

We then get
$$
\mu^i_j=\int \psi_ir_jd\mu=\int p_ip_jd\mu=\begin{cases} \Vert p_i\Vert ^2, &\mbox{if $j=i$},\\ 0,&\mbox{if $j\not =i$.}\end{cases}
$$
This gives
\begin{equation}\label{qnitm}
q_n^i=\sum_{j=0}^n\mu^i_js_{n-j}=\Vert p_i\Vert ^2 s_{n-i}.
\end{equation}
Hence the determinant in the right hand side of (\ref{mrby}) has the Toeplitz structure
$$
\left| \begin{matrix}s_n(x)&s_{n-1}(x)&s_{n-2}(x)&\cdots &s_{n-m+2}(x)&s_{n-m+1}(x)\\
s_{n+1}(x)&s_{n}(x)&s_{n-1}(x)&\cdots &s_{n-m+3}(x)&s_{n-m+2}(x)\\
\vdots&\vdots&\vdots & &\vdots &\vdots\\
s_{n+m-1}(x)&s_{n+m-2}(x)&s_{n+m-3}(x)&\cdots &s_{n+1}(x)&s_{n}(x)\\
\end{matrix}\right| .
$$
\bigskip

For some families of orthogonal polynomials, the polynomials
$q_n^i$ (\ref{qnitm}), $n,i\ge 0$, are, up to
change of variable, actually polynomials of the same family (but with different parameters). Among these families are Laguerre,  ultraspherical, Charlier and Meixner polynomials.

\subsubsection{Charlier polynomials}
Since $\Delta(\hat  c_n^a(x))=n\hat c_{n-1}^a(x)$, the determinants in (\ref{mrby}) for the operator $T_\mu$ are essentially the same that for the first order operator $\Delta$ (see Section \ref{ss5.2.1}).

\subsubsection{Laguerre polynomials}
The generating function for the Laguerre polynomials (\ref{deflap}) is
$$
\sum_{n=0}^\infty L_n^\alpha (x)t^n=(1-t)^{-\alpha-1}e^{xt/(t-1)}.
$$
Using (\ref{gfdi}) we get that $s_n(x)=L^{-\alpha-2}_n(-x)$, and, hence, taking into account (\ref{qnitm}) we deduce that
also $q_n^i(x)=\Vert L_i^\alpha\Vert ^2L^{-\alpha-2}_{n-i}(-x)$.

\subsubsection{Ultraspherical polynomials}\label{ss5.3.3}
Consider the ultraspherical (o Gegenbauer) polynomials $(C_n^\lambda)_n$, $\lambda \not =0$, defined by the generating function
\begin{equation}\label{ufg}
(1-2xz+z^2)^{-\lambda}=\sum_{n=0}^\infty C_n^\lambda(x)z^n.
\end{equation}
If $\lambda =-k$, with $k$ a positive integer, we have $C_n^{-k}=0$, $n\ge 2k$. It turns out that the finite sequence of polynomials $C_n^{-k}$, $n=0,\cdots , k$, satisfies a three term recurrence relation and hence they are orthogonal with respect to a (signed) measure.

For $\lambda=0$, the ultraspherical polynomials reduce to the Chebyshev polynomials of the first kind defined by the generating function
$$
\frac{1-xz}{1-2xz+z^2}=\sum_{n=0}^\infty C_n^0(x)z^n.
$$
The ultraspherical family $(C_n^\lambda)_n$ is the particular case of Jacobi polynomials for $\alpha=\beta=\lambda-1/2$.

Using (\ref{gfdi}) we get that $s_n(x)=C^{-\lambda}_n(x)$, $\lambda \not=0$ and, hence, taking into account (\ref{qnitm}) we deduce that also $q_n^i=\Vert C_i^\lambda\Vert ^2C^{-\lambda}_{n-i}(x)$. In particular, for $\lambda=k$ with $k$ a positive integer, we get $s_n=0$, $n\ge 2k+1$.

For the Chebyshev polynomials of the first kind, it is easy to see that $s_0=1$, $s_1(x)=-x$ and $s_n(x)=x^{n-2}(1-x^2)$, $n\ge 2$. And for the  Chebyshev polynomials of the second kind, we have $s_0=1$, $s_1(x)=-2x$, $s_2(x)=1$ and $s_n(x)=0$, $n\ge 3$. As a consequence $\det \left( q_{n+i-1}^{j-1}\right) _{i,j=1}^m=0$ for $n\ge 3$ and $m\ge 2$. Theorem \ref{mti} then implies that for $(p_n)_n$ the Chebyshev polynomials of the first or second kind, also $\det \left( p_{m+j-i}(x))\right) _{i,j=1}^n=0$ for $n\ge 3$ and $m\ge 2$ (where we take $p_n=0$ for $n<0$). This
is, anyway, an easy consequence of the fact that Chebyshev polynomials of the first and second kind have, essentially, constant recurrence coefficients.

\subsubsection{Meixner polynomials}\label{5.m.4}
For $a\not = 0,1$, consider the Meixner polynomials defined by the generating function
$$
\sum_{n=0}^\infty m_n^{a,c} (x)t^n=\left(1-\frac{t}{a}\right)^{x}(1-t)^{-x-c}.
$$
A straightforward computation shows that
$$
m_n^{a,c}(x)=\frac{(a-1)^n}{a^nn!}\hat m_n^{a,c}(x),
$$
where $\hat m_n^{a,c}(x)$ is the $n$-th monic Meixner polynomial defined by (\ref{mpol}).

Using (\ref{gfdi}) we get that $s_n(x)=m^{a,-c}_n(-x)$, and, hence, taking into account (\ref{qnitm}) we deduce that
also $q_n^i=\Vert m_i^{a,c}\Vert ^2m^{a,-c}_{n-i}(-x)$.

\section{Selberg type integrals and sums}
Write $\xx$ for the multivariable $\xx =(x_1,\cdots ,x_m)$ and $\Lambda (\xx )$ for the Vandermonde determinant:
\begin{equation}\label{defVd2}
\Lambda (\xx )=\prod_{1\le i<j\le m}(x_i-x_j).
\end{equation}
Given a measure $\mu$, consider the multiple integral
\begin{equation}\label{sti}
\int \cdots \int s(\xx )\Lambda^{2\gamma}(\xx )d\mu(\xx ),
\end{equation}
where $d\mu(\xx )=d\mu(x_1)\cdots d\mu(x_m)$ and $s(\xx )$ is a symmetric polynomial in the variable $x_1,\cdots , x_m$.

When $\mu$ is the Jacobi  measure $d\mu =x^{\alpha -1}(1-x)^{\beta -1} dx$, $\alpha , \beta >0$, and $s(\xx )=1$, the integral (\ref{sti}) is the celebrated Selberg integral (\cite{Se}, see also \cite{FW}):
$$
\int _0^1 \cdots \int_0^1 \Lambda^{2\gamma}(\xx ) \prod_{j=1}^mx_j^{\alpha -1}(1-x_j)^{\beta -1}d\xx=\prod_{j=0}^{m-1}\frac{\Gamma(\alpha+j\gamma)\Gamma(\beta+j\gamma)
\Gamma(1+(j+1)\gamma)}{\Gamma(\alpha+\beta +(m+j-1)\gamma)\Gamma(1+\gamma)}.
$$
As we wrote in the Introduction, other choices of $s(\xx )$ give well-known extensions of the Selberg integral. For instance, $s(\xx )=\prod_{j=1}^m(x_j-u)$ is Aomoto's integral \cite{Ao}, and $s(\xx )=J_\lambda(x_1,\cdots, x_m,1/\gamma)$, where $\gamma \in \NN $ and
$J_\lambda$ is a Jack polynomial, is Kadell's integral \cite{Ka2}. For $s(\xx )=1$ and $\mu$ the Hermite measure $d\mu =e^{-x^2/2}dx$, the integral (\ref{sti}) is the also celebrated Mehta integral.

In this Section, we prove Theorem \ref{tise} in the Introduction. That is, we
show that for certain operators $T$ the determinant in the right hand side of (\ref{mrby}) can be rewritten as an integral of the form (\ref{sti}) for $\gamma=1$ and $s(\xx )=\prod_{j=1}^mr_n(x_j)$. Hence, the determinant in the left hand side of (\ref{mrby}) gives the value of this integral.

The main idea is to use a variant of the Heine's integral formula (\cite{Hei}) for a Hankel matrix: If $\mu$ is a measure on the real line and we write
$\mu_j=\int x^jd\mu $, $j\ge 0$, for its moments, then
$$
\det ((\mu_{i+j-2})_{i,j=1}^n)=\int \cdots \int \Lambda^{2}(\xx )d\mu(\xx ).
$$
We need the following lemma
\begin{lemma}\label{llhi} Let $\mu$ be a measure in the real line and consider a sequence of polynomials $(\psi_n)_n$ with $\psi_n$ of degree $n$ and leading coefficient equal to $\upsilon_n$. Then, for any symmetric polynomial $s(\xx )$ we have
\begin{equation}\label{lhi}
\idotsint s(\xx )\Lambda(\xx )\prod_{j=1}^m\psi_{j-1}(x_j)d\mu(\xx )=\frac{(-1)^{\binom{m}{2}}\prod_{j=0}^{m-1}\upsilon_j}{m!}\idotsint s(\xx )\Lambda^{2}(\xx )d\mu(\xx ).
\end{equation}
\end{lemma}

\begin{proof}
For each permutation $\sigma$ of $\{1,\cdots ,m\}$, we perform in the integral in the left hand side of (\ref{lhi}) the change of variable $y_i=x_{\sigma(i)}$. Since $s$ is a symmetric polynomial and $\Lambda (\yy )=(-1)^{\epsilon(\sigma)}\Lambda(\xx )$, where $\epsilon (\sigma)$  denotes the signature of the permutation $\sigma$, we get
$$
\idotsint s(\xx )\Lambda(\xx )\prod_{j=1}^m\psi_{j-1}(x_j)d\mu(\xx )=
(-1)^{\epsilon(\sigma)} \idotsint s(\xx )\Lambda(\xx )\prod_{j=1}^m\psi_{j-1}(x_{\sigma (j)})d\mu(\xx ).
$$
Summing up over all permutations we find
\begin{align*}
\idotsint s(\xx )\Lambda(\xx )\prod_{j=1}^m\psi_{j-1}(x_j)d\mu(\xx )&=\frac{1}{m!}
\idotsint s(\xx )\Lambda(\xx )\sum_{\sigma}(-1)^{\epsilon(\sigma)}\prod_{j=1}^m\psi_{j-1}(x_{\sigma (j)})d\mu(\xx )\\
&=\frac{1}{m!}\idotsint s(\xx )\Lambda(\xx )\det((\psi_{j-1}(x_i))_{i,j=1}^m)d\mu(\xx ).
\end{align*}
The latter determinant can be
easily transformed into a Vandermonde determinant by elementary column
operations (see \cite{Kra}, Prop. 1) to get
$\det((\psi_{j-1}(x_i))_{i,j=1}^m)=(-1)^{\binom{m}{2}}\left( \prod_{j=0}^{m-1}\upsilon_j\right) \Lambda (\xx)$.
The proof is now finished.
\end{proof}

Given a number $u$ and an operator $T$, we associate to $T$ the polynomials $r_n,s_n$, $n\ge 0$, as in (\ref{defrn}), (\ref{defsn}). In this case we take $r_n(u)=0$, $n\ge 1$ (these conditions define $r_n$, $n\ge 0$, uniquely). We write $r_n=r_{n,u}$ to stress the dependence of $r_n$ on $u$. As an easy consequence we have that also $s_n(u)=0$, $n\ge 1$, (see (\ref{defsn})). This gives that $q_n^i(u)=\mu ^i_n=\int\bar \psi_ir_nd\mu$, where $q_n^i$ are the polynomials given by (\ref{defqny}). Hence the determinant in the right hand side of (\ref{mrby}) evaluated at $x=u$ can be rewritten as
\begin{align}\nonumber
\det\left( q^{j-1}_{n+i-1}(u)\right)_{i,j=1}^m&=\det\left( \int \bar\psi_{j-1}(x_j)r_{n+i-1,u}(x_j)d\mu(x_j)\right)_{i,j=1}^m \\ \label{hfpx}&=
\int \cdots \int \left(\prod_{j=1}^m\bar\psi_{j-1}(x_j)\right)
\det\left( r_{n+i-1,u}(x_j)\right)_{i,j=1}^m d\mu(\xx ).
\end{align}

We are now ready to prove Theorem \ref{tise} in the Introduction.

\begin{proof}
We first prove that
\begin{equation}\label{hfp}
\det\left( q^{j-1}_{n+i-1}(u)\right)_{i,j=1}^m=\int \cdots \int \frac{\Lambda ^2(\xx)}{m!}\prod_{j=1}^{m}\frac{\bar\upsilon_{j-1}\sigma_{n+j-1}}{\sigma_n}r_{n,u}(x_j)
d\mu(\xx ),
\end{equation}
where $\upsilon_n, \sigma_n$ are the leading coefficients of the polynomials $\psi_n$ and $r_n$, respectively.
Since $r_{n,u}(x)=\sigma_n\prod_{j=1}^n(x-a_j)$, we get
\begin{equation}\label{hfpxx}
\det\left( r_{n+i-1}(x_j)\right)_{i,j=1}^m=\det\left(h_{i-1}(x_j)\right)_{i,j=1}^m\prod_{j=1}^{m}\frac{\sigma_{n+j-1}}{\sigma_n}r_{n,u}(x_j)
\end{equation}
where $h_i(x)=\prod_{j=n+1}^{n+i}(x-a_j)$. The determinant in the right hand side of (\ref{hfpxx}) can be
easily transformed into a Vandermonde determinant by elementary column
operations (see \cite{Kra}, Prop. 1) to get
$\det\left(h_{i-1}(x_j)\right)_{i,j=1}^m=(-1)^{\binom{m}{2}}\Lambda (\xx)$.
Hence (\ref{hfpx}) and (\ref{hfpxx}) give
$$
\det\left( q^{j-1}_{n+i-1}(u)\right)_{i,j=1}^m=(-1)^{\binom{m}{2}}\int \cdots \int \left(\prod_{j=1}^{m}\frac{\sigma_{n+j-1}}{\sigma_n}r_{n,u}(x_j)\right)
\Lambda (\xx)\prod_{j=1}^m\bar\psi_{j-1}(x_j) d\mu(\xx ).
$$
(\ref{hfp}) now follows easily using Lemma \ref{llhi}.
The proof of the Theorem can be finished by using Theorem \ref{mti} and the formula (\ref{oitnm}) for $\Omega _n$ in terms of the norm of the monic orthogonal polynomials.
\end{proof}

The  relevant operators $T=d/dx$ and $T=\Delta$ satisfies the hypothesis of Theorem \ref{tise} with
$$
r_{n,u}(x)=\frac{(x-u)^n}{n!},\quad \mbox{and $\displaystyle r_{n,u}(x)=\binom{x-u}{n}$,}
$$
respectively.
Hence, we have

\begin{corollary}
Let $\mu$ be a measure and write $(\hat p_n)_n$ for the sequence of monic orthogonal polynomials with respect to $\mu$. Then
\begin{align*}\label{stfdu1}
\int \cdots \int \Lambda ^2(\xx )\prod_{j=1}^m(x_j-u)^n d\mu(\xx )&=C_{n,m}\det\left( \hat p_{m+j-1}^{(i-1)}(u)\right)_{i,j=1}^n,\\
\int \cdots \int \Lambda ^2(\xx )\prod_{j=1}^m(x_j-u-n+1)_n d\mu(\xx )&=C_{n,m}\det\left( \Delta^{i-1}(\hat p_{m+j-1}(u))\right)_{i,j=1}^n,
\end{align*}
where
$$
C_{n,m}=\frac{(-1)^{mn}m!\prod_{j=0}^{m-1}\Vert \hat p_j\Vert^2}{\prod_{j=0}^{n-1}j!}.
$$
\end{corollary}

When $\mu$ is one of the the classical weights of Jacobi, Laguerre or Hermite, and $u=0$, the determinant in the right hand side of the first identity of the previous Corollary can be easily computed using standard determinant techniques. In particular, for $\mu$ the Jacobi weight and $u=0$, we recover the Selberg integral for $\gamma=1$. For $\mu$ the Jacobi weight and $n=1$ we recover Aomoto's result for (\ref{sti}) when $\gamma=1$ (\cite{Ao}).
In a similar way, for the classical discrete polynomials we get a deduction for Askey conjectures 5, 6 and 7 in \cite{Ask} when $\gamma=1$.

\bigskip

Theorem \ref{tise} can be rewritten only in terms of the polynomials $r_n$.

\begin{corollary}\label{corf} Let $(a_n)_n, (\sigma _n)_n$ be two sequences of numbers with $\sigma _0=1$, $\sigma_n\not =0$, $n\ge 1$. Write $r_{n}$, $n\ge 0$, for the polynomials of degree $n$ defined by $r_0=1$, $r_{n}(x)=\sigma_n\prod_{j=1}^n(x-a_j)$, $n\ge 1$. Consider a  measure $\mu$ and its associated sequence $(\hat p_n)_n$ of monic orthogonal polynomials. For each $n\ge 0$, we can write $\hat p_n(x)=\sum_{j=0}^n \theta^n_jr_j(x)$. Then
\begin{equation}\label{hfpif}
\int \cdots \int \Lambda ^2(\xx)\prod_{j=1}^{m}r_{n}(x_j)
d\mu(\xx )=C_{T,n,m}\det\left( \theta ^{m+j-1}_{i-1}\right)_{i,j=1}^n,
\end{equation}
where
$$
C_{T,n,m}=(-1)^{mn}m!\sigma_n^m\prod_{j=0}^{m-1}\Vert \hat p_j\Vert^2\prod_{j=0}^{n-1}\sigma_{j}.
$$
\end{corollary}

\begin{proof}
Define the operator $T$ by linearity from $T(r_n)=r_{n-1}$, $n\ge 1$, and $T(1)=0$.
Then $T^i(\hat p_n)(x)=\sum_{j=0}^n\theta ^n_jr_{j-i}(x)$, and since $r_{n}(a_1)=0$, $n\ge 1$, we get $T^i(p_n)(a_1)=\theta ^n_i$. It is enough now to apply Theorem \ref{tise}.
\end{proof}

Here it is an illustrative example of how to use this Corollary to get a Selberg type sum for Racah polynomials.

\subsection{Racah polynomials}\label{Racsec}
Let $N,\beta,\gamma, \delta$ be real numbers satisfying that $N$ is a positive integer and $\delta ,\gamma ,-\beta+\gamma \not = -1,\cdots ,-N$,
$\beta +\delta \not =-1,\cdots , -2N-1$. Define the numbers $\sigma_n$ and the polynomials $\tau_n, r_n$, $0\le n \le N$, by
\begin{align*}
\sigma_n&=\frac{(-1)^n}{(-N)_n(\beta+\delta+1)_n(\gamma+1)_nn!},\\
\tau_n(x)&=(-1)^n\prod_{i=0}^{n-1}(x-i(\gamma +\delta +1+i)),\\
r_n(x)&=\sigma_n(-1)^n\tau_n(x).
\end{align*}
Write $\lambda(x)=x(x+\gamma+\delta+1)$ and for $0\le n\le N$ consider the monic Racah polynomials (see \cite{KLS}, pp. 190-196) defined by
\begin{equation}\label{defrc}
\hat R_n^{N,\beta,\gamma,\delta}(\lambda(x))=\frac{(-1)^n}{n!\sigma_n(\beta+n-N)_n}\pFq{4}{3}{-n,n+\beta-N,-x,x+\gamma+\delta+1}
{-N,\beta+\gamma+1,\gamma+1}{1}.
\end{equation}

The relationship of the polynomials  $(\hat R_n^{N,\beta,\gamma,\delta})_n$ with the Racah polynomials $(r_n^{\alpha ,\beta,\gamma,\delta})_n$ defined in Section
\ref{5.4.2rw} is
$$
\hat R_n^{N,\beta,\gamma,\delta}(x)=n!r_n^{N,\beta,\gamma,\delta}(x).
$$
The polynomials $(\hat R_n^{N,\beta,\gamma,\delta})_n$ are orthogonal with respect to the discrete measure
$$
\sum_{x=0}^N\frac{(-N)_x(\beta+\delta+1)_x(\gamma+1)_x(\gamma+\delta+1)_x((\gamma+\delta+3)/2)_x}
{N +\gamma+\delta+2)_x(-\beta+\gamma+1)_x((\gamma+\delta+1)/2)_x(\delta+1)_xx!}\delta_{\lambda(x)}.
$$
The norm of the $n$-th monic Racah polynomial is
$$
\Vert \hat R_n^{N,\beta,\gamma,\delta}\Vert ^2=\frac{(-1)^nM(\beta-\gamma -N)_n(-\delta-N)_n(\beta+1)_n}{\sigma_n(\beta-N+1)_{2n}(\beta-N+n )_n},
$$
where $M=\frac{(-\beta)_N(\gamma+\delta+2)_N}{(-\beta+\gamma+1)_N(\delta +1)_N}$.

A simple computation shows that $\tau_n(\lambda(x))=(-x)_n(x+\gamma+\delta+1)_n$. Hence, (\ref{defrc}) can then be rewritten as
$$
\hat R_n^{N,\beta,\gamma,\delta}(\lambda(x))=\frac{(-1)^n}{n!\sigma_n(\beta+n-N)_n}\sum_{j=0}^n(-n)_j(n+\beta-N)_jr_j(x)
$$
With the notation of Corollary \ref{corf}, we have
$$
\theta^n_j=\frac{(-1)^n(-n)_j(n+\beta-N)_j}{n!\sigma _n(\beta+n-N)_n}.
$$
Using standard determinant techniques we get
$$
\det \left( \theta ^{m+j-1}_{i-1}\right)_{i,j=1}^n=(-1)^{mn}\prod_{j=0}^{n-1}\frac{j!}{(m+j)!\sigma_{m+j}(\beta+m+j-N)_{m}}.
$$
Hence Corollary \ref{corf} gives the identity
\begin{align*}
\sum_{x1,\cdots, x_m=n}^N&\Lambda ^2(\xx _\lambda )\prod_{j=1}^m\frac{(-N,\beta+\delta+1,\gamma+1,\gamma+\delta+3)/2)_{x_j}(\gamma+\delta+1)_{x_j+n}}
{(N +\gamma+\delta+2,-\beta+\gamma+1,(\gamma+\delta+1)/2,\delta+1)_{x_j}(x_j-n)!}\\&=
(-1)^{mn}M^m\prod_{j=0}^{m-1}\frac{(j+1)!(\beta -\gamma-N,-\delta -N,\beta+1)_j(-N,\beta+\delta+1,\gamma+1)_{n+j}}
{(\beta -N+1)_{2j}(\beta-N+j)_j(\beta-N+m+j)_n},
\end{align*}
where $\xx _\lambda=(\lambda(x_1),\cdots, \lambda(x_m))$, and $(u_1,\cdots ,u_k)_j=(u_1)_j\cdots (u_k)_j$.

Similar versions can be get for the other Racah weights (see \cite{KLS}, pp. 190).

The above identity for Racah polynomials is a particular case of Theorem 3.1 of \cite{Schl}
(which it is itself a special case of a sum that was originally conjectured by Ole Warnaar \cite{Waa}
and subsequently proved by Hjalmar Rosengren \cite{HRos}) which gives an extension of Frenkel and Turaev's ${}_{10}V_9$ summation formula:
it is the specialization $p\to 0$, followed by $d\to\infty$, and then $m\mapsto N-n$, $a\mapsto q^{n+(\gamma+\delta+1)/2}$,
$b\mapsto \beta+\delta+1$, $c\mapsto \gamma+1$, and finally $q\to 1$.

\section{Constant term identities}
In this Section, we show that for certain operators $T$ and certain families of orthogonal polynomials $(p_n)_n$, one (or both) of the determinants in (\ref{mrby}) can be rewritten as the constant term of certain multivariate Laurent expansions.
Some of them are generalizations of Morris identity (\ref{mocti}) for $k=1$. We next consider illustrative examples of this for Meixner, Charlier and ultraspherical polynomials.

\subsection{Meixner polynomials}
For $a\not = 0,1$, consider as in Section \ref{5.m.4} the Meixner polynomials $(m_n^{a,c})_n$ defined by the generating function
\begin{equation}\label{mfg}
\left( 1-\frac{z}{a}\right) ^x(1-z)^{-x-c}=\sum_{n=0}^\infty m_n^{a,c}(x)z^n.
\end{equation}
Summing up rows we have
\begin{equation}\label{mds7.2}
\det\left(\Delta^{i-1}(m_{m+j-1}^{a,c}(x))\right)_{i,j=1}^n=
\det\left(m_{m+j-1}^{a,c}(x+i-1)\right)_{i,j=1}^n.
\end{equation}
Using Theorem \ref{mti} and taking into account the computations in Section \ref{ss5.2.2} we get
\begin{equation}\label{midu}
\det\left(\Delta^{i-1}(m_{m+j-1}^{a,c}(x))\right)_{i,j=1}^n=C_{n,m}\det\left(m_{n+i-1}^{a,-c-n-i-j+3}(-x+j-1)
\right)_{i,j=1}^m,
\end{equation}
where
$$
C_{n,m}=\frac{(-1)^{nm+\binom{m}{2}+\binom{n}{2}}(1-a)^{\binom{n}{2}}}{a^{\binom{n}{2}-\binom{m}{2}}}.
$$
The generating function (\ref{mfg}) gives for the Meixner polynomials the integral representation
\begin{equation}\label{mieup}
m_n^{a,c}(x)=\frac{1}{2\pi \imath }\int_C\frac{\left( 1-\frac{z}{a}\right) ^x(1-z)^{-x-c}}{z^{n+1}}dz,
\end{equation}
where the function $f_x(z)=\left( 1-\frac{z}{a}\right) ^x(1-z)^{-x-c}$ is taken so that it is analytic at $z=0$ and $f_x(0)=1$, and $C$ is a circle with center at $z=0$, positively oriented and such that $f_x(z)$ is analytic in a neighborhood of $C$ and in its interior.
Using this integral representation, we can rewrite the determinant in the left hand side of (\ref{midu}) in terms of the constant term in the Laurent expansion of certain multivariate analytic function.

Indeed, using (\ref{mieup}) the determinant in the right hand side of (\ref{mds7.2}) can be rewritten as
$$
\frac{1}{(2\pi \imath)^n}\int _C \cdots \int _C\det\left(\frac{\left( 1-\frac{z_j}{a}\right) ^{x+i-1}(1-z_j)^{-x-i-c+1}}{z_j^{m+j}}\right)_{i,j=1}^n d\zz ,
$$
where $d\zz =dz_1\cdots dz_n$. This multiple integral can be computed as follows
\begin{align*}
\int _C \cdots \int _C&\det\left(\frac{\left( 1-\frac{z_j}{a}\right) ^{x+i-1}(1-z_j)^{-x-i-c+1}}{z_j^{m+j}}\right)_{i,j=1}^n d\zz \\&=
\int _C \cdots \int _C\prod_{j=1}^n\frac{\left( 1-\frac{z_j}{a}\right)^{x}(1-z_j)^{-x-c}}{z_j^{m+j}}\det\left(\frac{\left( 1-\frac{z_j}{a}\right)^{i-1}}{(1-z_j)^{i-1}}\right)_{i,j=1}^n d\zz \\&=(-1)^{\binom{n}{2}}\int _C \cdots \int _C\prod_{j=1}^n\frac{\left( 1-\frac{z_j}{a}\right)^{x}}{(1-z_j)^{x+c}z_j^{m+j}}\prod_{1\le i<j\le n}\frac{(1-a)(z_i-z_j)}{a(1-z_i)(1-z_j)} d\zz \\&=
\frac{(a-1)^{\binom{n}{2}}}{n!a^{\binom{n}{2}}}\int _C \cdots \int _C\prod_{j=1}^n\frac{\left( 1-\frac{z_j}{a}\right)^{x}}{(1-z_j)^{x+c+n-1}z_j^{m+n}}\prod_{1\le i<j\le n}(z_i-z_j)^2 d\zz,
\end{align*}
where we have used the version of Lemma \ref{llhi} for a contour integral.

The Residue Theorem then gives that the last integral is equal to
$$
(2\pi\imath)^n\mbox{C.T.$\vert _{\zz =0}$} \prod_{j=1}^n\frac{\left( 1-\frac{z_j}{a}\right)^{x}}{(1-z_j)^{x+c+n-1}z_j^{m+n-1}}\prod_{1\le i<j\le n}(z_i-z_j)^2,
$$
where if $F(\zz )$ is a multivariate meromorphic function $F(\zz )$ at $z_1=\cdots=z_n=0$, we denote by $\mbox{C.T.$\vert _{\zz =0}$}F(\zz )$ the constant term in the Laurent expansion at $z_1=\cdots=z_n=0$ of $F(\zz )$.

Using (\ref{midu}), we finally get
\begin{align*}
\mbox{C.T.$\vert _{\zz =0}$} \prod_{j=1}^n&\frac{\left( 1-\frac{z_j}{a}\right)^{x}}{(1-z_j)^{x+c+n-1}z_j^{m+n-1}}\prod_{1\le i<j\le n}(z_i-z_j)^2\\ &=(-1)^{\binom{m}{2}+\binom{n}{2}+nm}n!a^{\binom{m}{2}} \det\left(m_{n+i-1}^{a,-c-n-i-j+3}(-x+j-1) \right)_{i,j=1}^m.
\end{align*}

\subsection{Charlier polynomials}
Consider the Charlier polynomials $(c_n^{a})_n$ defined by the generating function
\begin{equation}\label{chfg}
( 1+z) ^xe^{-az}=\sum_{n=0}^\infty c_n^{a}(x)z^n.
\end{equation}
A straightforward computation shows that
$$
c_n^{a}(x)=\hat c_n^{a}(x)/n!,
$$
where $\hat c_n^{a}(x)$ is the $n$-th monic Charlier polynomial defined by (\ref{chpol}).

Proceeding as in the previous Section (using Theorem \ref{mti} and the computations in Section \ref{ss5.2.1}),
we get
\begin{equation}\label{chidu}
\det\left(c_{m+j-1}^{a}(x+i-1)\right)_{i,j=1}^n=(-1)^{nm}\det\left(c_{n+i-1}^{-a}(-x+j-1)
\right)_{i,j=1}^m.
\end{equation}
The generating function (\ref{chfg}) gives for the Charlier polynomials the integral representation
\begin{equation}\label{chieup}
c_n^{a}(x)=\frac{1}{2\pi \imath }\int_C\frac{(1+z)^xe^{-az}}{z^{n+1}}dz,
\end{equation}
where the function $f_x(z)=(1+z)^xe^{-az}$ is taken so that it is analytic at $z=0$ and $f_x(0)=1$, and
$C$ is a circle with center at $z=0$, positively oriented and such that $f_x(z)$ is analytic in a neighborhood of $C$ and its interior.

Using this integral representation and proceeding as in the previous Section, we can rewrite the determinant in the left hand side of (\ref{chidu}) in terms of the constant term in the Laurent expansion of certain multivariate analytic function. In doing that, we get
$$
\mbox{C.T.$\vert _{\zz =0}$} \prod_{j=1}^ne^{-az_j}\frac{(1+z_j)^{x}}{z_j^{m+n-1}}\prod_{1\le i<j\le n}(z_i-z_j)^2=
(-1)^{\binom{n}{2}+nm}n!\det\left(c_{n+i-1}^{-a}(-x+j-1)\right)_{i,j=1}^m.
$$

\subsection{Ultraspherical polynomials I}
Consider, as in Section \ref{ss5.3.3}, the ultraspherical (o Gegenbauer) polynomials $(C_n^\lambda)_n$, $\lambda \not =0,-1,-2,\cdots $, defined by the generating function
\begin{equation}\label{ufg2}
(1-2xz+z^2)^{-\lambda}=\sum_{n=0}^\infty C_n^\lambda(x)z^n.
\end{equation}
Consider the operator $T (C_n^\lambda)=C_{n-1}^\lambda$. Using Theorem \ref{mti} and the computations in Section \ref{ss5.3}
and \ref{ss5.3.3} we get
\begin{equation}\label{idu}
\det(C_{m+j-i}^\lambda(x))_{i,j=1}^n=(-1)^{nm}\det(C_{n+i-j}^{-\lambda}(x))_{i,j=1}^m.
\end{equation}
The generating function (\ref{ufg2}) gives for the ultraspherical polynomials ($\lambda \not =0$) the integral representation
\begin{equation}\label{ieup}
C_n^\lambda(x)=\frac{1}{2\pi \imath }\int_C\frac{(1-2xz+z^2)^{-\lambda}}{z^{n+1}}dz,
\end{equation}
where the function $f_x(z)=(1-xz+z^2)^{-\lambda}$ is taken so that it is analytic at $z=0$ and $f_x(0)=1$, and
$C$ is a circle with center at $z=0$, positively oriented and such that $f_x(z)$ is analytic in a neighborhood of $C$ and in its interior. Notice that if we write $C_n^\lambda(x)=0$ for $n<0$, the integral representation (\ref{ieup}) also holds for $n<0$.

Using this integral representation and proceeding as in the previous Section, we can rewrite  the determinant in the left hand side of (\ref{idu}) in terms of the constant term in the Laurent expansion of certain multivariate analytic function. In doing that, we get for $\lambda\not =0,-1,-2, \cdots $
\begin{equation}\label{umf}
\mbox{C.T.$\vert _{\zz =0}$} \prod_{j=1}^n\frac{(1-2xz_j+z_j^2)^{-\lambda}}{z_j^{m+n-1}}\prod_{1\le i<j\le n}(z_i-z_j)^2=
(-1)^{\binom{n}{2}+nm}n!\det(C_{n+j-1}^{-\lambda}(x))_{i,j=1}^m.
\end{equation}
Actually, this formula holds for $\lambda\not =0$. Indeed, for $\lambda =-k$, with $k$ a positive integer, the formula (\ref{ufg2}) allows to
define the polynomials $C_n^{-k}$, $n\ge 0$. For $n\ge 2k$ we have $C_n^{-k}=0$. It is not difficult to see that for fixed $n,m$, the determinant in the right hand side of (\ref{umf}) is a polynomial in $\lambda$ vanishing at $\lambda =0$. Since the left hand side of (\ref{umf}) is a continuous function in $\lambda$ and they are equal for $\lambda\not =0,-1,-2, \cdots $, they are actually equal for all values of $\lambda \not =0$. For $\lambda =0$, the left hand side of (\ref{umf}) is then 0 (this happens because the Chebyshev polynomials of the first kind $C_n^0$ do not satisfy the generating function (\ref{ufg2})).

For $x=0$, the determinant in the right hand side of (\ref{umf}) can be explicitly computed. In doing that, one has
\begin{align*}
\mbox{C.T.$\vert _{\zz =0}$}& \prod_{j=1}^n\frac{(1+z_j^2)^{-\lambda}}{z_j^{m+n-1}}\prod_{1\le i<j\le n}(z_i-z_j)^2
\\ & = (-1)^{\binom{n}{2}+nm}n!
\begin{cases}
0,& \mbox{if $nm$ is odd,}\\ \omega (n,m,\lambda ),&\mbox{if $nm$ is even and $n\le m$,}
\\ \omega (m,n,-\lambda ),&\mbox{if $nm$ is even and $m\le n$,}
\end{cases}
\end{align*}
where for $n\le m$ we define
$$
\omega (n,m,\lambda )=(-1)^{n\left[ \frac{m+1}{2}\right] }\prod_{j=0}^{n-1}\frac{[j/2]!}{[\frac{m+j}{2}]!}\prod_{j=0}^{[\frac{m+1}{2}]-1}(\lambda +j)^{\min (n,m-2j)}(\lambda -j-1)^{\max (n-2j-2,0)}.
$$

\subsection{Ultraspherical polynomials II}
If we consider the derivative instead of the operator $T$ considered in the previous Section,
the integral representation (\ref{ieup}) leads to other interesting constant term identity.

For ultraspherical polynomials we have the formula
$$
(C_n^\lambda)'(x)=2\lambda C_{n-1}^{\lambda+1}(x).
$$
This gives
\begin{equation}\label{ds7.2}
\det\left((C_{m+j-1}^\lambda (x))^{(i-1)}\right)_{i,j=1}^n=2^{\binom{n}{2}}\left(\prod_{i=1}^n(\lambda)_{i-1}\right)\det\left(C_{m+j-i}^{\lambda +i-1}(x)\right)_{i,j=1}^n.
\end{equation}
Using Theorem \ref{mti} and the computations in Section \ref{ss5.1} we get for $\lambda \not = -1/2,-1,-3/2,-2,\cdots $,
\begin{equation}\label{idu2}
\det\left((C_{m+j-1}^\lambda (x))^{(i-1)}\right)_{i,j=1}^n=C_{n,m}\det\left(\frac{(1-x)^{n+i-1}P_{n+i-1}^{2\lambda+j-1,-n-i-\lambda+1/2}\left(\frac{x+3}{x-1}\right)}
{(2\lambda+2j-1)_{n+i-j}}\right)_{i,j=1}^m,
\end{equation}
where
$$
C_{n,m}=(-1)^{nm+\binom{m}{2}}2^{\binom{n+m}{2}}\frac{\prod_{j=1}^{n+m-1}(\lambda)_j}{\prod_{j=1}^{m-1}(2\lambda)_j}.
$$
Using the integral representation (\ref{ieup}), and proceeding as in the previous Section we can write the determinant in the left hand side of (\ref{idu2}) in terms of the constant term in the Laurent expansion of certain multivariate analytic function. In doing that, we get for $\lambda \not = -1/2,-1,-3/2,-2,\cdots $,
\begin{align*}
&\mbox{C.T.$\vert _{\zz =0}$} \prod_{j=1}^n\frac{(1-2xz_j+z_j^2)^{-\lambda-n+1}}{z_j^{m}}\prod_{1\le i<j\le n}(z_i-z_j)^2\left(1-\frac{1}{z_iz_j}\right)\\ &\quad =
(-1)^{nm+\binom{m}{2}}n!2^{nm+\binom{m}{2}}\left(\prod_{j=0}^{m-1}\frac{(\lambda)_{n+j}}{(2\lambda)_j}\right)
\det\left(\frac{(1-x)^{n+i-1}P_{n+i-1}^{2\lambda+j-1,-n-i-\lambda+1/2}\left(\frac{x+3}{x-1}\right)}
{(2\lambda+2j-1)_{n+i-j}}\right)_{i,j=1}^m.
\end{align*}
For $x=1$, the determinant in the right hand side of the above identity can be explicitly computed to get
\begin{align*}
\mbox{C.T.$\vert _{\zz =0}$}& \prod_{j=1}^n\frac{(1-z_j)^{-2\lambda-2n+2}}{z_j^{m}}\prod_{1\le i<j\le n}(z_i-z_j)^2\left(1-\frac{1}{z_iz_j}\right)\\ &=
\frac{n!2^{2nm+\binom{m}{2}}\left(\prod_{j=1}^{m-1}(2\lambda+2j-1)^{m-j}\right)}{\left(\prod_{j=1}^{[m/2]}(2\lambda+n+2j-1)_{[\frac{m-1}{2}]+
[\frac{m+1}{2}]}\right)}
\prod_{j=0}^{m-1}\frac{j!(\lambda)_{n+j}(\lambda +j+1/2)_n}{(n+j)!(2\lambda)_j(2\lambda+2j+1)_n}.
\end{align*}

\bigskip
\noindent
\textit{Acknowledgement} The author would like to thank Doron Zeilberger for pointing him out  that the conjecture
\ref{c1.1} was related to certain Selberg and Morris type identities, and to M. Schlosser for pointing him out the connection of the Racah identity
in Section \ref{Racsec} to an extension of Frenkel and Turaev's ${}_10V9$ summation formula.

     \end{document}